\documentclass[
reqno, 
12pt]{amsart}
\usepackage{amssymb,amsmath, amsthm,latexsym, mathrsfs}

\usepackage{a4wide}

\usepackage{hyperref}

\usepackage[all]{xy}

\usepackage{tikz}

\theoremstyle{plain}

\newtheorem{lemma}{Lemma}[section]
\newtheorem{theorem}[lemma]{Theorem}
\newtheorem{proposition}[lemma]{Proposition} 
\newtheorem{corollary}[lemma]{Corollary}

\newtheorem{remark}[lemma]{Remark}

\newtheorem{definition}[lemma]{Definition}

\parskip=\bigskipamount

\font\rm=cmr12

\def\Z{\mathbb Z}
\def\R{\mathbb R}
\def\d{\delta}

\def\t{\times}
\def\o{\otimes}

\def\h{\hookrightarrow}
\def\ra{\rightarrow}
\def\la{\leftarrow}

\def\mt{\mapsto}

\def\a{\alpha}

\def\z{\mathfrak z}

\def\D{\Delta}
\def\G{\Gamma}

\def\s{\sigma}

\def\SC{ \mathcal S(\mathcal C)}

\def\Der{\mathscr D}

\def\C{\mathcal }

\title[Reducibility the points beyond the ends of complementary series]
{On the reducibility points beyond the ends of complementary series of $p$-adic general linear groups}

\mark{Marko Tadi\'c}

\author{Marko Tadi\'c}

\address{Department of Mathematics, University of Zagreb
\\
Bijeni\v{c}ka 30, 10000 Zagreb,
 Croatia\\
Email: \tt tadic{\char'100}math.hr}


\thanks{The   
author was partly supported by 
Croatian Ministry of Science, Education and Sports grant
{\#}037-0372794-2804.}
\date{\today}

\begin{document}

\begin{abstract} In this paper we  consider the reducibility points beyond the ends of complementary series    of  general linear groups over a $p$-adic field, which start  with  Speh representations. We  describe explicitly  the composition series of the representations  at these reducibility points. They  are multiplicities  one representations, and they can be of arbitrary length.   We give Langlands parameters of all the irreducible subquotients.

\end{abstract}

\maketitle

\setcounter{tocdepth}{1}

%\begin{centerline}
%{--------- \ \ \ \  Preliminary version \ \ \ \  ---------}
%\end{centerline}

%\tableofcontents

\section{Introduction}\label{intro}

Problems of reducibility of parabolically induced representations are very important  in the harmonic analysis on reductive groups over local fields (they are of particular importance for the problem of unitarizability). They are also very important for the theory of automorphic forms for number of questions.
% (among others, see section 6 of \cite{T-harm} for a "reducibility" interpretation of the local packets constructed in \cite{A-book}). 
A closely related (usually very non-trivial) problem is the determination of the composition series at the   reducibility points. The knowledge of composition series is equivalent to the corresponding character identity. In this paper we study such type of problems for general linear groups over a 
local non-archimedean field $F$.
 
  Speh representations  are key representations in the classification of unitary duals of general linear groups (see \cite{T-AENS}; for the archimedean case see \cite{T-R-C-old}). 
One directly gets all the complementary series for general linear groups  from the complementary series starting with Speh representations.
Composition series at the ends of these complementary series are crucial in determining the topology of the unitary duals.
The composition series at the ends of these complementary series played also crucial  role in obtaining explicit formula for characters of irreducible unitary representations in terms of standard characters (\cite{T-ch}). These complementary series terminate at  the first reducibility point and the representations there have length two  (when one starts with a single Speh representation). It is a natural question to ask what are the composition series at the further reducibility points.  

There can exist a significant number of reducibility points beyond the end of complementary series. In all these reducibility points we completely determine the composition series, and give the Langlands parameters of the irreducible subquotients. These representations are always multiplicity one representations (when one starts from a single Speh representation), and can be of arbitrary length. For example, if we want to get a representation of length $1000$ supported by the minimal parabolic subgroup, we shall need to start with Speh representations of $GL(999\ 000,F)$\footnote{This is the lowest rank in which we get length 1000 at some reducibility point beyond the complementary series starting with a single Speh representation.}, and consider the complementary series of $GL(1\ 998\ 000,F)$. It is interesting that it is very easy to write down the Langlands parameters of all the irreducible subquotients  (they are given by the simple formulas \eqref{first}; see  Theorem \ref{th-i}).

Now we shall describe more precisely the principal results of the paper.
Put
$$
\nu=|\det|_F,
$$
where $|\ |_F$ denotes the normalized absolute value on a local non-archimedean field $F$.
For $u,v\in\mathbb R$ such that $v-u\in\Z_{\geq0}$,  and for an irreducible cuspidal representation $\rho$ of $GL(p,F)$,
the set 
$$
[\nu^u\rho,\nu^v\rho]=\{\nu^{u}\rho, \nu^{u+1}\rho,\dots,\nu^{v-1}\rho,\nu^{v}\rho\}
$$
 is called a segment in cuspidal representations (of general linear groups). The representation  $\nu^u\rho$  
 is denoted by $b([\nu^u\rho,\nu^v\rho])$, and 
 called the beginning  
 of the segment $[\nu^u\rho,\nu^v\rho]$. We say that such a segment $\D_1$ precedes another segment $\D_2$, and write
 $$
 \D_1\ra \D_2,
 $$
 if $\D_1\cup \D_2$ is a segment different from $\D_1$ and $\D_2$, and if the beginnings of $\D_1$ and $\D_1\cup \D_2$ are the same.
 
Let  $\D=
[\nu^u\rho,\nu^v\rho]$ be a segment in cuspidal representations. For $z\in \mathbb R$, we denote 
 $$
 \nu^z\D=\{\nu^z\rho';\rho'\in \D\}.
 $$

 Consider the representation
$$
\text{Ind}^{GL((v-u+1)p,F)}(\nu^v\rho\o\nu^{v-1}\rho\o\dots\o\nu^u\rho),
$$
parabolically induced from the appropriate  parabolic subgroup containing regular upper triangular matrices (see the second section).
Then the above representation has a unique irreducible subrepresentation. This sub representation  is essentially square integrable. It is  denoted by
$$
\d(\D).
$$
Let $a=(\D_1,\dots,\D_k)$ be a finite  multiset of segments  in cuspidal representations. Write $\D_i=
[\nu^{u_i}\rho_i,\nu^{v_i}\rho_i]$, where $\d(\D_i)$ is a representation of $GL(n_i,F)$ and $\rho_i$ are unitarizable irreducible cuspidal representations. Take a permutation $\s$ of $\{1,\dots,k\}$ such that
\begin{equation}
\label{ord}
u_{\s(1)}+v_{\s(1)}\geq \dots\geq u_{\s(k)}+v_{\s(k)}.
\end{equation}
Then the representation
$$
\text{Ind}^{GL(n_1+\dots+n_k,F)}(\d(\D_{\s(1)})\o\dots\o\d(\D_{\s(k)})),
$$
parabolically induced from the appropriate  parabolic subgroup containing regular upper triangular matrices, has a unique irreducible quotient (whose equivalence class does not depend on the permutation $\s$ which satisfy the above condition\footnote{The multiset does not change if we make a permutation of  elements in it. Nevertheless, when we  define multisets, we  always fix some ordering on the segments that determine  it (often in this paper completly opposite to \eqref{ord})}). We denote it by 
$$
L(a).
$$
Then attaching $a\mapsto L(a)$ is one possible description of the Langlands classification of the non-unitary duals of groups $GL(n,F)$'s (by multisets of segments  in cuspidal representations). We shall use this version of  Langlands classification for general linear groups in this paper. 

We add finite  multisets of segments in obvious way: 
$$
a_1+a_2=(\D_1^{(1)},\dots,\D_{k^{(1)}}^{(1)},\D_1^{(2)},\dots,\D_{k^{(2)}}^{(2)}),
$$
 where
$a_i=(\D_1^{(i)},\dots,\D_{k^{(i)}}^{(i)})$, $i=1,2$.

Let $\D=[\nu^u\rho,\nu^v\rho]$ be a segment in cuspidal representations such that $\rho$ is unitarizable and $u+v=0$. Fix $n\in\Z_{\geq 1}$. Then the representation
$$
u(\d(\D),n):=L(\nu^{(n-1)/2}\D,\nu^{(n-1)/2-1}\D,\dots, \nu^{-(n-1)/2}\D)
$$
is called a Speh representation. Such a representation is unitarizable,  and each irreducible unitary representation of a general linear group is constructed from several  such representations in a simple way (\cite{T-AENS}). If an irreducible representation become a Speh representations after a twists   by a characters, then it will be called   essentially Speh representation. In other words, essentially Speh representations are the representations of a form  $\nu^\a u(\d(\D),n)$, with $\a\in \mathbb R$.

Let $k\in \Z$ and denote $d=\text{card}(\D)$. We shall consider representations
$$
\mathbf R^{\mathbf t}(n,d)_{(k)}^{(\rho)}=\text{Ind}(\nu^{-k/2}u(\d(\D),n)\o \nu^{k/2}u(\d(\D),n)).
$$
For $k=1$, this is the end of complementary series.
Denote 
$$
\D_i=\nu^{-k/2}(\nu^{-(n-1)/2+i-1}\D), \qquad \G_i=\nu^{k/2}(\nu^{-(n-1)/2+i-1}\D), \qquad i=1,\dots,n.
$$
Then
$$
\nu^{-k/2}u(\d(\D),n)=L(\D_1,\dots,\D_n), \qquad \nu^{k/2}u(\d(\D),n)=L(\G_1,\dots,\G_n),
$$
and thus
$$
\mathbf R^{\mathbf t}(n,d)_{(k)}^{(\rho)}=\text{Ind}(L(\D_1,\dots,\D_n) \o L(\G_1,\dots,\G_n)).
$$

\begin{definition} 
For $j=0$ and for  $1\leq j\leq n$ for which $\D_n\ra \G_j$, denote
\begin{equation}
\label{first}
r_j(n,d)_{(k)}^{(\rho)}=\sum_{i=1}^j(\D_{i+n-j}\cup \G_i, \D_{i+n-j}\cap \G_i)+\sum_{i=1}^{n-j}(\D_i,\G_{i+j}).
\end{equation}
\end{definition}

In other words,
$
r_0(n,d)^{(\rho)}_k=(\D_{1},\dots,\D_n,\G_1,\dots,\G_n)$, and for
 $1\leq j\leq n$ for which $\D_n\ra \G_j$ we get
$
r_j(n,d)^{(\rho)}_k
$
 by replacing
in
$
r_0(n,d)^{(\rho)}_k=(\D_{1},\dots,\D_n,\G_1,\dots,\G_n)
$
the part
$$
\D_{n-j+1},\dots,
\D_n,\G_1,
\dots,\G_j
$$
with
$$
\D_{n-j+1}\cup \G_1,\D_{n-j+1}\cap \G_1,
\dots\dots,
\D_{n}\cup\G_j,\D_n\cap\G_j.
$$
One gets easily that $\D_n\ra \G_j$ if and only if
$$
\max( n-k+1,1)\leq j \leq \min(n-k+d,n).
$$

\begin{theorem}
\label{th-i}
 Let $k\in\Z_{\geq 0}$.  Then:
\begin{enumerate}

\item The representation $\mathbf R^{\mathbf t}(n,d)_{(k)}^{(\rho)}$ is a multiplicity one representation. It has a unique irreducible subrepresentation and unique irreducible quotient. The irreducible subrepresentation is isomorphic to $L(r_0(n,d)_{(k)}^{(\rho)})$,
Further, $\mathbf R^{\mathbf t}(n,d)_{(k)}^{(\rho)}$ and $\mathbf R^{\mathbf t}(n,d)_{(-k)}^{(\rho)}$ have the same composition series.

\item
For 
$
 n+d\leq k
 ,
 $
 and for $k=0$,  $\mathbf R^{\mathbf t}(n,d)_{(k)}^{(\rho)}$ is irreducible. Then $\mathbf R^{\mathbf t}(n,d)_{(k)}^{(\rho)}\cong L(r_0(n,d)_{(k)}^{(\rho)})$.

\item
 For 
 $
0 < k < n+d,
$
the composition series of $\mathbf R^{\mathbf t}(n,d)_{(k)}^{(\rho)}$ consists of
$$
L(r_i(n,d)_{(k)}^{(\rho)}), \quad \max(n-k+1,1)\leq i \leq \min(n-k+d,n),
$$
together with $L(r_0(n,d)_{(k)}^{(\rho)})$.
 The irreducible quotient of $\mathbf R^{\mathbf t}(n,d)_{(k)}^{(\rho)}$ is isomorphic  to $L(r_{\min(n-k+d,n)}(n,d)_{(k)}^{(\rho)}).$

\end{enumerate}

\end{theorem}

Besides the composition series, in this paper we also completely determine the lattices of subrepresentations of representations $\mathbf R^{\mathbf t}(n,d)_{(k)}^{(\rho)}$.

The main tool in our handling of the composition series that we consider in this paper, are derivatives (\cite{GK}, \cite{BZ}, \cite{Z}). A very simple and nice  formula for the derivatives of Speh representations (described in 2.10) is crucial for our applications of derivatives. The formula was obtained in \cite{LMi} (conjectured much earlier in \cite{T-der-ber}). Actually, the formula of E. Lapid and A. M\'inguez is much more general - it is   for ladder representations (defined in \cite{LMi}\footnote{Let $\pi\cong L(a)$, where
$a=(\D_1,\dots,\D_k)$ is a finite multiset of segments in cuspidal representations. 
Suppose that $\D_i\not\subseteq \D_j$ whenever  $i\ne j$. If $\pi$ is supported by one cuspidal $\Z$-line (this is not essential condition), then $\pi$ is called a ladder representation.}). 
Another tool that we use in this paper is the M\oe glin-Waldspurger algorithm from \cite{MoeW-alg} for the Zelevinsky involution.

Although the main results of the paper are presented in the introduction in terms of the Langlands classification, the methods by which we have obtained them in the paper are based on the Zelevinsky classification (which is dual to the Langlands classification). We deal in the most of the paper with the representations $\mathcal R(n,d)_{(k)}^{(\rho)}$ which are defined in terms of the Zelevinsky classification. The relation with the representations $\mathbf R^{\mathbf t}(n,d)_{(k)}^{(\rho)}$ which we describe in the introduction is very simple
$$
\mathbf R^{\mathbf t}(n,d)_{(k)}^{(\rho)}=\mathcal R(d,n)_{(k)}^{(\rho)}.
$$
The main reason why the Zelevinsky classification is much more convenient for working with the derivatives, is that  there is a very simple explicit formula for the highest derivative of an irreducible representation in terms of this classification (Theorem 8.1 of \cite{Z};  in the case of Langlands classification we have an  algorithm).

We are very thankful to B. Leclerc who has informed us that our result about composition series in the case when the cuspidal representation $\rho$ is the trivial character of $F^\t$  can be deducted from Theorem 2 of \cite{Le}, which addresses Hecke algebra representations (his result is more general in this case - it gives combinatorial rule for calculating the
 composition factors there).
The theory of types for general linear groups from \cite{BK1}
 %(together with  \cite{BK2}), 
 opens a possibility of approach to get the case of general $\rho$ using B. Leclerc result. We have not used this possibility. This way of proving the general case would be  technically  more complicated, relaying on types, attached Hecke algebras etc. (and already in the unramified case, it is not simple since \cite{Le}  is based on two previous papers,   one of which uses a very non-trivial positivity result of G. Lusztig).

The main reason for our approach is that the statement of our principal result does not include types, and therefore it is natural to (try to) have a proof of it which does not use them. 
The second reason are derivatives which we use in this paper (and  develop further methods for applying them). They are a very natural  tool in the study of questions related to the irreducible unitary representations. Namely, recall that already the main result of the first crucial  paper \cite{Be-P-inv} on the unitarizability in the $p$-adic case,  
relates unitarizability and derivatives for general linear groups. This J.~Bernstein paper was followed by the second paper \cite{T-AENS} where the unitarizability was solved completely, with essential use of derivatives (and soon realized in \cite{T-R-C-old} that the solution can be extended to the archimedean situation, avoiding derivatives in this case).  At the end, let us mention that our experience with the problems related to the unitary representations,  is that it is very important to have as simple (and direct) understanding of them  as possible.

We are very thankful to the referee for a number of corrections and very useful suggestions.

The content of the paper is as follows. The second section recalls  the notation that we  use in the paper. The third section contains preparatory technical results, while in the fourth section we define the multisegments $r_i(n,d)_{(k)}^{(\rho)}$. The fifth is devoted to the calculation of the composition series in the case of disjoint beginnings of segments, while the sixth section gives a description of the lattice of subrepresentations in this case. In  the seventh section we deal with the composition series in the remaining case (of non-disjoint  beginnings of segments) and the eighth section brings a description of the lattice of subrepresentations for this case. The ninth  section brings an interprettin in terms of the Langlands classification  of the main results of the paper (which are obtained in previous sections in terms of the Zelevinaky classification). At the end of this section we present  a conjectural  description of the composition series of the representation parabolically induced with a tensor product of  two arbitrary essentially Speh representations.

\section{Notation and preliminaries}

  We  recall very briefly some notation for general linear groups in the non-archimedean case (one can find  more details in \cite{Z} and \cite{Ro}). 
  
  \subsection{Finite multisets}
  
   Let $X$ be a set. The set of all finite multisets in $X$ is denoted by $M(X)$ (we can view each multiset as a functions $X\rightarrow \Z_{\geq0}$ with finite support; here  finite subsets correspond to all functions $\C S(\C C)\rightarrow \{0,1\}$ with finite support).
Elements of $M(X)$ are denoted by $(x_1,\dots,x_n)$ (repetitions of elements can occur; the  multiset does not change  if we permute $x_i$'s). The number 
$
n
$
 is called the cardinality of $(x_1,\dots,x_n)$, and it is denoted by
 $$
\text{card}(x_1,\dots,x_n).
$$
On $M(X)$ we have   a natural structure of a commutative associative semi group with zero: 
$$
(x_1,\dots,x_n)+(y_1,\dots,y_m)=(x_1,\dots,x_n,y_1,\dots,y_m).
$$
For $x,y \in M(X)$ we write 
$$
x\subseteq y
$$
 if 
 there exists $z\in M(X)$ such that $x+z=y$.

\subsection{Segments in $\C C$} 

Let $F$ be a non-archimedean locally compact non-discrete field and $|\ \ |_F$ its modulus character. 
Denote 
$$
G_n=GL(n,F), n\geq 0
$$
(we take $G_0$  to be the trivial group; we consider it formally  as the group of $0\times 0$ matrices).
The set of all equivalence classes of irreducible  representations of all groups $G_n, n\geq 0$, is denoted by
$$
Irr.
$$
The subset of all cuspidal classes of representations $G_n, n\geq 1$, is denoted by
$$
\C C.
$$
Unitarizable classes in $\C C$ are denoted by $\C C^u$.
 Put
$$
\nu=|\det|_F.
$$
Fix $u,v\in\mathbb \R$ such that $v-u\in \Z_{\geq 0}$,  and  $\rho\in \C C$. Then
the set 
$$
[\nu^u\rho,\nu^v\rho]=\{\nu^{u}\rho, \nu^{u+1}\rho,\dots,\nu^{v-1}\rho,\nu^{v}\rho\}
$$
 is called a segment in $\C C$. The set of all segments in cuspidal representations of general linear groups is denoted by
$$
\SC.
$$
Let  $\D=
[\nu^u\rho,\nu^v\rho]\in \SC$. 
The representation  $\nu^u\rho$   
 is called the beginning  
 of the segment $\D$, and  $\nu^v\rho$   
 is called the end  
 of the segment $\D$. We denote the beginning and the end by
 $$
 b(\D) \text{\ \ and \ \ } e(\D)
 $$
 respectively.

For $z\in \mathbb R$,  denote 
 $$
 \nu^z\D=\{\nu^z\rho';\rho'\in \D\}.
 $$
 We define $\D^-$ and $^-\D$ by 
 $$
 \D^-=
[\nu^{u}\rho,\nu^{v-1}\rho]
\text{ and }
 ^-\D=
[\nu^{u+1}\rho,\nu^v\rho]
$$ 
if $u<v$. Otherwise we take $^-\D=\emptyset$.

Segments $\D_1,\D_2 \in \SC$ are called linked if $\D_1\cup\D_2\in\C S(\C C)$ 
and $\D_1\cup\D_2\not\in\{\D_1,\D_2\}$. If the segments $\D_1$ and $\D_2$ are linked and if $\D_1$ and $\D_1\cup\D_2$ have the same beginnings, we say that 
$\D_1$ precedes $\D_2$. In this case we write
$$
\D_1\rightarrow \D_2.
$$

 \subsection{Multisegments} 

Let $a=(\D_1,\dots,\D_k)\in M(\SC)$. Suppose that $\D_i$ and $\D_j$ are linked for some $1\leq i<j\leq k$. Let $c$ be the multiset that we get by replacing segments $\D_i$ and $\D_j$  by segments $\D_i\cup \D_j$ and $\D_i\cap \D_j$ in $a$ (we omit $\D_i\cap \D_j$ if  $\D_i\cap \D_j=\emptyset$). In this case we write
$$
c\prec a.
$$
Using $\prec$, we generate in a natural way an ordering $\leq $ on $M(\SC)$.

For $a=(\D_1,\dots,\D_k)\in M(\SC)$,  denote
$$
a^-=(\D_1^-,\dots,\D_k^-), \     ^-a=(^-\D_1,\dots,^-\D_k)\  \in\   M(\SC)
$$
 (again, we omit $\D_i^-$ and $^-\D_i$ if they are empty sets).

 Further, for $a=(\D_1,\dots,\D_k)\in M(\SC)$  define
$$
\text{supp}(a)=\sum_{i=0}^k \D_i\in M(\mathcal C),
$$
where we consider in the above formula $\D_i$'s as elements of $M(\C C)$. 

The multiset of all beginnings (reps. ends) of segments from $a\in M(\SC)$ is denoted by $\C B(a)$ (resp. $\C E(a)$). Clearly, 
$$
\C B(a),\C E(a)\in M(\C C).
$$
 
 Take positive integers $n$ and $d$ and let $\rho\in \C C$. Denote
 \begin{equation}
\label{and}
a(n,d)^{(\rho)}=(\nu^{-\frac{n-1}2}\D,\nu^{-\frac{n-1}2+1}\D,\dots ,\nu^{\frac{n-1}2}\D)\in M(\SC),
\end{equation}
where  
$$
 \D
 =
 [\nu^{-(d-1)/2}\rho,\nu^{(d-1)/2}\rho].
 $$

 \subsection{Algebra of representations}

 The category of all smooth representations of $G_n$ is denoted by Alg($G_n)$. The set of all equivalence classes of irreducible smooth representations of $G_n$ is denoted by 
 $$
 \tilde G_n.
 $$
  The subset of unitarizable classes in $\tilde G_n$ is denoted by 
  $$
  \hat G_n.
  $$
The Grothendieck group of the category Alg$_{\text{f.l.}}(G_n)$ of all smooth representations of $G_n$ of finite length is denoted by $R_n$. It is a free $\mathbb Z$-module with basis $\tilde G_n$.
We have the canonical mapping  
$$
\text{s.s.}: \text{Alg}_{\text{f.l.}}(G_n)\ra R_n.
$$
The set of all finite sums in $R_n$ of elements of the basis $\tilde G_n$ is denoted by $(R_n)_+$. Set
$$
\aligned
&R=\oplus_{n\in \mathbb Z_{\geq0}} R_n,
\\
&R_+=\sum_{n\in \mathbb Z_{\geq0}} (R_n)_+.
\endaligned
$$
The ordering on $R$ is defined by $r_1\leq r_2 \iff r_2-r_1\in R_+$. 

An additive mapping $\varphi:R\ra R$ is called positive if
$$
r_1\leq r_2 \implies \varphi(r_1)\leq \varphi(r_2 ).
$$

For two finite length representations $\pi_1$ and $\pi_2$ of $G_n$ we shall write $\text{s.s.}(\pi_1)\leq \text{s.s.}(\pi_2)$ shorter
$$
\pi_1\leq \pi_2.
$$

 \subsection{Parabolic induction}
Let 
$$
M_{(n_1,n_2)}:=\left\{
\left[
\begin{matrix}
g_1 & *
\\
0 & g_2
\end{matrix}
\right]
;
g_i\in G_{i}
\right\} \subseteq G_{n_1+n_2},
$$
and let $\sigma_1$ and $\sigma_2$ be smooth representations of $G_{n_1}$ and $G_{n_2}$, respectively.
We consider $\s_1\o\s_2$ as a the  representation 
$$
\left[
\begin{matrix}
g_1 & *
\\
0 & g_2
\end{matrix}
\right]
\mapsto
\s_1(g_1)\o\s_2(g_2)
$$
of $M_{(n_1,n_2)}$.
By 
$$
\s_1\t\s_2
$$
is denoted the representation of $G_{n_1+n_2}$ parabolically induced by $\s_1\o\s_2$ from $M_{(n_1,n_2)}$ (the induction that we consider  is smooth and normalized). For three representations, we have
\begin{equation}
\label{asso}
(\s_1\t\s_2)\t\s_3\cong \s_1\t(\s_2\t\s_3).
\end{equation}
The induction functor is exact and we can lift it in a natural way to a $\mathbb Z$-bilinear mapping $\t:R_{n_1}\t R_{n_2}\rightarrow R_{n_1+n_2}$, and further to $\t : R\t R\rightarrow R$. In this way, $R$ becomes graded commutative ring. The commutativity implies that if $\pi_1\t\pi_2$ is  irreducible for $\pi_i\in\tilde G_{n_i}$, then
$$
\pi_1\t\pi_2\cong
\pi_2\t\pi_1.
$$

\subsection{Classifications of non-unitary duals} Let $\D=[\nu^u\rho,\nu^v\rho]\in\SC$. The representation
$$
\nu^v\rho\t\nu^{v-1}\rho\t\dots\t\nu^u\rho.
$$
 has a unique irreducible subrepresentation, which is denoted by
$$
\d(\D),
$$
and a unique irreducible quotient, which is denoted by 
$$
\z(\D).
$$
 The irreducible subrepresentation  is essentially square integrable, i.e. it becomes square integrable (modulo center) after twisting with a suitable character of the group.

Let $a=(\D_1,\dots,\D_n)\in M(\SC)$ be non-empty (i.e. $n\geq 1$). 
Choose an  enumeration of $\D_i$'s 
such that for all $i,j\in \{1,2,\dots,n\}$ the following  holds:

\begin{center}
if $\D_{i}\rightarrow \D_{j} $,  
then $j<i$.
\end{center}
Then the representations
$$
\aligned
&\zeta(a):= \z(\D_1)\t \z(\D_2)\t\dots\t \z(\D_n),
\\
& \lambda(a):= \d(\D_1)\t \d(\D_2)\t\dots\t \d(\D_n)
\endaligned
$$
are determined by $a$ up to an isomorphism (i.e., their isomorphism classes do not depend on the enumerations which satisfies the above condition). The representation $\zeta(a)$ has a unique irreducible subrepresentation, which is denoted by 
$$
Z(a),
$$
 while the representation $\lambda(a)$ has a unique irreducible quotient, which is denoted by 
 $$
 L(a).
 $$
 For the empty multisegment $\emptyset$, we take $Z(\emptyset)=L(\emptyset)$ to be the trivial (one-dimensional) representation of the trivial group $G_0$. This is the identity of the ring $R$ (and it is very often denoted simply by $1$).
 
  In this way we  obtain two classifications of $Irr$ by $M(\SC)$.
Here, $Z$ is called Zelevinsky classification of Irr, while $L$ is called Langlands classification of Irr.

 It is well known (see \cite{Z}) that for $a,b\in M(\SC)$ holds
 \begin{equation}
 \label{cs}
 Z(b)\leq \zeta(a) \iff b\leq a.
\end{equation}
The contragredient representation of $\pi$ is denoted by $\tilde \pi$. For $\D\in \SC$, set $\tilde \D:=\{\tilde \rho;\rho\in \D\}$. If $a=(\D_1,\dots,\D_k)\in M(\SC)$, then we put
$
\tilde a=(\tilde \D_1,\dots,\tilde \D_k).
$
Then 
$$
L(a)\tilde{\ }=L(\tilde a) \text{ \ and \ } Z(a)\tilde{\ }=Z(\tilde a).
$$
Analogous relations hold for Hermitian contragredients. The Hermitian contragredient of a representation $\pi$ is denoted by 
$$
\pi^+.
$$

\subsection{Classification of the unitary dual} 

Denote by 
$$
B_{\text{rigid}}=\{Z(a(n,d)^{(\rho)}); n,d\in\Z_{\geq1}, \rho\in \C C^u\}.
$$
 and
 $$
 B=B_{\text{rigid}}\cup\{\nu^\a \s\t\nu^{-\alpha}\s; \s\in B_{\text{rigid}}, 0<\a<1/2\}.
$$
 Then the unitary dual is described  by the following:

\begin{theorem} $($\cite{T-AENS}$)$
The map 
$$
(\tau_1,\dots,\tau_r)\mt \tau_1\t\dots\t\tau_r
$$
 is a bijection between $M(B)$ and the set of all equivalence classes of irreducible unitary representations of groups $GL(n,F)$, $n\geq 0$.

%\begin{enumerate}
%\item Let 
%$\tau_1, \ldots , \tau_n \in B$.
%Then the representation 
%$$
%\pi:=\tau_1 \times \ldots \times \tau_n
%$$
% is irreducible and unitary.
%
%\item Suppose that a representation
%$\pi'$
%is obtained from
%$\tau_1', \ldots , \tau_{n'}' \in B$
%in the same manner as
%$\pi$
%was obtained from
%$\tau_1, \ldots , \tau_n$
%in (1). Then
%$\pi \cong \pi'$
%if and only if $n=n'$ and if the sequences
%$(\tau_1, \ldots , \tau_n)$
%and
%$(\tau_1', \ldots , \tau_n')$
%coincide after a renumeration.
%
%\item Each irreducible  unitary representation of $G_m$, for
%any $m$, can be obtained as in (1). 
%\end{enumerate}
\end{theorem}

\subsection{Duality - Zelevinsky involution} Define a mapping
$$
^t:\text{Irr} \ra \text{Irr}
$$
by $Z(a)^t=L(a), a \in M(\SC)$.  Obviously,
\begin{equation}
\label{inv-def}
\z(\D)^t=\d(\D),\quad \D\in \SC.
\end{equation}
We lift $^t$ to an additive homomorphism $^t:R\ra R$.  Clearly, $^t$ is a positive mapping, i.e., satisfies: $r_1\leq r_2\implies r_1^t\leq r_2^t$.  A non-trivial fact is that $^t$ is also multiplicative, i.e., a ring homomorphism (see \cite{Au} and \cite{ScSt}\footnote{More precisely, A.V. Zelevinsky used \eqref{inv-def} to define the involution on $R$, and \cite{Au} and \cite{ScSt} prove the positivity of this involution. Our definition of $^t$ is equivalent to that of A.V. Zelevinsky (see the beginning of  section \ref{sec-conn-Z-L}).}). Further, $^t$ is an involution. Define $a^t\in M(\SC)$ for $ a \in M(\SC)$  by the requirement
$$
(L(a))^t=L(a^t).
$$
We could also use the Zelevinsky classification to define $^t:M(\SC)\ra M(\SC)$, and we would get the same involutive mapping.
Recall that
$$
Z(a_1+a_2)\leq Z(a_1)\t Z(a_2)
$$
(Proposition 8.4 of \cite{Z}).
 From this follows directly 
\begin{equation}
\label{tt}
Z((a_1^t+a_2^t)^t)\leq Z(a_1)\t Z(a_2).
\end{equation}
One can find more information about the involution in \cite{Ro}.

%We shall point out one useful property of this involution,  which is a consequence of the multiplicativety of $^t$. Let $P\in \Z[X_1,\dots,X_n]$\footnote{In applications in the paper,  $P$ will be very simple polynomial of the form $X_1X_2-X_3-\dots-X_n$.} and $a_1,\dots,a_n\in M(\SC)$. Then in $R$ holds
%\begin{equation}
%\label{eq-t}
%P(Z(a_1),\dots,Z(a_n))=0 \iff P(L(a_1),\dots,L(a_n))=0 .
%\end{equation}

\subsection{Algorithm of C. M\oe glin and J.-L. Waldspurger} 
Let $a\in M(\SC)$ be  non-empty. Fix $\rho\in \mathcal C$ and denote by
$$
X_\rho(a)
$$ 
the set of all $x\in\R$ such that there exists a segment $\D$ in $a$ satisfying $e(\D)\cong \nu^x\rho$.

Now fix $\rho$ such that $X_\rho(a)\ne \emptyset$. Let $x=\max (X_\rho(a))$, and
consider segments $\D$  in $a$ such that $e(\D)\cong \nu^x\rho$. Among these segments, choose one of minimal cardinality. Denote it by $\D_1$. This will be called the first stage of the algorithm.

Consider now segments $\D$ in $a$  such that $e(\D)\cong \nu^{x-1}\rho$, and which are linked with $\D_1$. Among them, if  this set is non-empty, choose one with minimal cardinality. Denote it by $\D_2$.

One continues this procedure with ends $x-2$, $x-3$, etc., as long as it is possible. The segments considered in this procedure are $\D_1, \dots, \D_k$ ($k\geq 1$).
Let 
$$
\G_1=[e(\D_{k}),e(\D_1)]=[\nu^{x-k+1}\rho,\nu^x\rho]\in  \SC.
$$

 Let $a^\la$ be the multiset of  $M(\C C)$ which we get from $a$ by replacing each $\D_i$ by $\D_i^-$, $i=1,\dots ,k$ (we  omit those $\D_i^-$ for which  $\D_i^-=\emptyset$). 
 
 If $a^\la$  is non-empty, we now repeat  the above procedure with $a^\la$ as long as possible. In this way we get a segment $\G_2$ and $(a^\la)^\la\in M(\SC)$.

Continuing this procedure as long as possible, until we reach the empty set, we get $\G_1,\dots,\G_m\in \mathcal S(\C C)$. Then by \cite{MoeW-alg}  
we have 
$$
a^t= (\G_1,\dots,\G_m).
$$
This algorithm will be denoted by
$$
\text{MWA}^{\la}.
$$

There is also a dual (or "left") version of this algorithm, denoted by $^\ra$MWA (see \cite{T-GL-red}).

With this, is is easy to show that
\begin{equation}
\label{t}
Z(a(n,d)^{(\rho)})\cong L(a(d,n)^{(\rho)})
\end{equation}
 for $n,d\in\Z_{\geq1}$ and $\rho\in\mathcal C$ (in \cite{T-AENS} is obtained this relation in a different way).

\subsection{Derivatives on the level of $R$}  The algebra $R$ is a $\Z$-polynomial algebra over $\{\z(\D);\D\in \SC\}$ (Corollary 7.5 of \cite{Z}). Therefore, there exists a unique ring homomorphism
$$
\Der:R\ra R 
$$
satisfying
$$
\Der(\z(\D))=\z(\D)+\z(\D^-),\quad \forall \D\in\SC.
$$
Let $r\in R_n$, $r>0$. Write $\Der(r)=\sum_i r^{(i)},$ where $r^{(i)}\in R_i$. Then obviously $r^{(i)}=0$ for $i> n$, and $r^{(n)}=r$. Denote by $k$ the maximal index satisfying $r^{(i)}=0$ for all $i<k$. Then we define the highest derivative $\text{h.d.}(r)$ of $r$ by
$$
\text{h.d.}(r)=r^{(k)}.
$$
Let $r_i\in R_{n_i}$, $r_i>0$, for $i=1,2$. Then obviously holds
$$
\text{h.d.}(r_1\t r_2)=\text{h.d.}(r_1)\t \text{h.d.}(r_2)
$$
since $R$ is a graded  integral domain.

We shall use the derivatives on the level of $R$ in most of the paper. Only in the sections 
\ref{sec-der-disjoint-b} and  
\ref{sec-der-lattice-non-disjoint-b} we shall use derivatives on the level of representations (where we study the lattices of subrepresentations). 
The following two fundamental very non-trivial facts about $\Der$  play  an important role in our paper\footnote{One can can find at the beginning of  section
 \ref{sec-der-disjoint-b}  the references for them.}:

\begin{enumerate}

\item
 $\Der$ is a positive homomorphism (i.e. $r> 0\implies \Der(r)> 0$).

\item
Let $\pi=Z(a)$ be an irreducible representation of $G_n$, and consider it as an element of $R$.  
%Let $\pi=Z(a)$, where $a\in M(\SC)$,  
%and let $Z(a^-)$ be a representation of $G_l$.
% Write $\Der(\pi)=\sum_{i\geq 0}\Der_i(\pi)$.
Then
$$
\text{h.d.}(Z(a))=Z(a^-).
$$
%and $\Der_{j}(Z(a))=0$ for $j<l$ (see \cite{Z}). The representation $Z(a^-)$ is called the highest derivative of $Z(a)$. 
\end{enumerate}

Observe that using the  above formula for the highest derivative of $Z(a)$, one easily reconstructs $Z(a)$ from its highest derivative and the cuspidal support of $a$.  
%One defines in a naturel way the highest derivative of a not-necessarily irreducible representation, like $Z(a_1)\t\dots\t Z(a_k)$. Then 
This implies that using the formula
$$
\text{h.d.}(Z(a_1)\t\dots\t Z(a_k))
=
Z(a_1^-)\t\dots\t Z(a_k^-),
$$
we can reconstruct from the composition series of the above highest derivative  all the irreducible subquotients $Z(a)$ of $Z(a_1)\t\dots\t Z(a_k)$ which satisfy
$$
\text{card}(a)= \text{card}(a_1+\dots+a_k).
$$
Moreover, the corresponding multiplicities also coincide.

E. Lapid and A. M\'inguez have obtained in \cite{LMi} the formula for the derivative of the ladder representations (ladder representations are defined in \cite{LMi}). Representations  $Z(a(n,d)^{\rho)})$ are very special case of ladder representations. We shall now explain this formula in the case of representations $Z(a(n,d)^{\rho)})$ (this formula  will play crucial role in our paper).
Write   $a(n,d)^{(\rho)}=(\D_1,\dots,\D_n)$ in a way that $\D_1\ra \dots\ra \D_n$. Then
$$
\Der(Z(\D_1,\dots,\D_n))= Z(\D_1,\dots,\D_n) \hskip85mm
$$
$$
 + Z(\D_1^-,\D_2,\dots,\D_n)+\dots+Z(\D_1^-,\dots,\D_{n-1}^-,\D_n)+Z(\D_1^-,\dots,\D_n^-).
$$

\begin{remark}
\label{rm-der-dual}
Sometimes it is useful to consider the positive ring homomorphism 
$$
\tilde{\ }\Der\tilde{\ }:R\ra R, \pi\mapsto (\Der(\tilde\pi))\tilde{\ }.
$$
This homomorphism has analogous properties as $\Der$: it is positive and it sends
$$
\z(\D)\mapsto\z(\D)+\z(^-\D).
$$
 Here the highest derivative of $Z(a)$ for this homomorphism is $Z(^-a)$. Further one  has
 $$
 \tilde{\ }\Der\tilde{\ } (Z(a(n,d)^{(\rho)}))
 = Z(\D_1,\dots,\D_n)   \hskip85mm
$$
$$
+ Z(\D_1,\dots,\D_{n-1},^-\D_n)+\dots+Z(\D_1,^-\D_2,\dots,^-\D_n)+Z(^-\D_1,\dots,^-\D_n).
$$
\end{remark}

\section{Some general technical lemmas}

\subsection{Representations}
We shall consider the representations
$$
{\C R}(n,d)_{(l)}^{(\rho)}:=\nu^{-l/2}Z(a(n,d)^{(\rho)})\times \nu^{l/2}Z(a(n,d)^{(\rho)})
$$
$$
\hskip20mm= Z(a(n,d)^{(\nu^{-l/2}\rho)})\times Z(a(n,d)^{(\nu^{l/2}\rho)}),
$$
where $n,d \in\mathbb Z_{\geq1}$, $l\in\Z$ and $\rho\in\C C$. 
Observe that in $R$ we have
\begin{equation}
\label{invo}
({\C R}(n,d)_{(l)}^{(\rho)})^t={\C R}(d,n)_{(l)}^{(\rho)}
\end{equation}
and
$$
\text{s.s.}({\C R}(n,d)_{(l)}^{(\rho)})=\text{s.s.}({\C R}(n,d)_{(-l)}^{(\rho)}).
$$
The formula for the highest derivative is
$$
\text{h.d.}({\C R}(n,d)_{(k)}^{(\rho)})={\C R}(n,d-1)_{(k)}^{(\nu^{-1/2}\rho)}
$$
(we take formally  ${\C R}(n,0)_{(k)}^{(\rho)}$ to be $Z(\emptyset)).$

In what follows, we shall always assume 
$$
k \in\Z_{\geq0}.
$$
When $n,d,k$ and $\rho$ are fixed, 
  we  denote
$$
a_-=a(n,d)^{(\nu^{-k/2}\rho)}, \quad a_+=a(n,d)^{(\nu^{k/2}\rho)}.
$$
Write segments $\D_1,\dots,\D_n$ of $a_-$ in a way that 
$$
\D_1\ra\D_2\ra\dots \ra \D_n,
$$
and segments $\G_1\dots,\G_n$ of $a_+$ also in a way that
$$
\G_1\ra\G_2\ra\dots \ra \G_n.
$$
Now we introduce the following numbers (in $\frac12\Z)$:

$
\qquad \quad A_-=-\frac{d-1}2-\frac{n-1}2-\frac k2,\qquad \quad\quad\qquad\qquad B_-=\frac{d-1}2-\frac{n-1}2-\frac k2,
$

$
\qquad \qquad  C_-=-\frac{d-1}2+\frac{n-1}2-\frac k2,\qquad \quad\quad\qquad\qquad D_-=\frac{d-1}2+\frac{n-1}2-\frac k2,
$

$
 \qquad \quad \qquad A_+=-\frac{d-1}2-\frac{n-1}2+\frac k2,\quad\qquad \qquad  \qquad \quad B_+=\frac{d-1}2-\frac{n-1}2+\frac k2,
$

$
\qquad  \quad \quad \qquad C_+=-\frac{d-1}2+\frac{n-1}2+\frac k2,\quad\qquad \qquad  \qquad \quad\quad  D_+=\frac{d-1}2+\frac{n-1}2+\frac k2.
$

Observe that
$$
\D_1=[\nu^{A_-}\rho,\nu^{B_-}\rho], \quad \D_2=[\nu^{A_-+1}\rho,\nu^{B_-+1}\rho], \quad \dots, \quad
\D_n=[\nu^{C_-}\rho,\nu^{D_-}\rho],
$$
$$
\G_1=[\nu^{A_+}\rho,\nu^{B_+}\rho], \quad \G_2=[\nu^{A_++1}\rho,\nu^{B_++1}\rho], \quad \dots, \quad
\G_n=[\nu^{C_+}\rho,\nu^{D_+}\rho].
$$
Obviously, $A_-\leq  B_-,C_-\leq D_-$, $A_-\leq C_-,B_-\leq D_-$
and
$
B_--A_-=D_--C_-.
$
Analogous relations hold for $A_+,B_+,C_+$ and $D_+$.

It is well known that ${\C R}(n,d)_{(0)}^{(\rho)}$ is irreducible (see \cite{Ba-Sp}, and also \cite{T-AENS}).

Observe that for $D_-+2\leq A_+$, ${\C R}(n,d)_{(k)}^{(\rho)}$ is irreducible by Proposition 8.5 of \cite{Z}.
In other words, for 
$$
n+d\leq k
$$
${\C R}(n,d)_{(k)}^{(\rho)}$ is irreducible. Therefore, we can have reducibility (for $k\geq 0$) only if 
\begin{equation}
\label{eq-1-diff}
1\leq k \leq n+d-1.
\end{equation}

 We shall assume in the rest of this section that   \eqref{eq-1-diff} holds.

  Consider exponents that show up in the cuspidal supports of both $a_-$ and $a_+$. The cardinality of this set is $D_--A_++1$, which is
\begin{equation}
\label{intersection}
n+d-1-k.
\end{equation} 

\subsection{Unique irreducible subrepresentation and quotient}
\begin{proposition} 
\label{prop-sub-quo}
Suppose 
$$
1\leq k \leq n+d-1.
$$
The representation 
$$
{\C R}(n,d)_{(-k)}^{(\rho)}
$$
 has a unique irreducible subrepresentation, and it is isomorphic to
\begin{equation}
\label{eq-sub-rep}
Z(a(n,d)^{(\nu^{-k/2}\rho)}+a(n,d)^{(\nu^{k/2}\rho)}).
\end{equation}
Further, it has a unique irreducible quotient, and it is isomorphic to
\begin{equation}
\label{eq-quot}
Z((a(d,n)^{(\nu^{-k/2}\rho)}+a(d,n)^{( \nu^{k/2}\rho)})^t). 
\end{equation}
Both irreducible representations have multiplicity one in the whole representation.

The representation ${\C R}(n,d)_{(k)}^{(\rho)}$  has \eqref{eq-quot} as   unique irreducible subrepresentation and \eqref{eq-sub-rep}  as unique irreducible quotient. Their position is opposite to the position in ${\C R}(n,d)_{(-k)}^{(\rho)}$.
 \end{proposition}
 
 \begin{proof} Let $t$ be any integer satisfying $0\leq t \leq n$.  We have
 $
  Z(a_+) \h \zeta(a_+)
 $
 and 
 $$
  Z(\G_{n-t+1},\dots,\G_n)\t Z(\G_1,\dots,\G_{n-t})\h \zeta(a_+).
 $$
 Since $ \zeta(a_+)$ has a unique irreducible subrepresentation, and it is $Z(a_+)$, we get that there exists an embedding 
 $$
 Z(a_+)\h Z(\G_{n-t+1},\dots,\G_n)\t Z(\G_1,\dots,\G_{n-t}).
 $$
 Analogously, we get
 $$
 Z(a_-)\h Z(\D_{t+1},\dots,\D_n)\t Z(\D_1,\dots,\D_t).
 $$
 Now we shall specify $t$. If
 $$
 C_-+D_-\leq A_++B_+
 $$
 (i.e. $n-1\leq k$), we take $t=n$. In the opposite case $A_++B_+<C_-+D_-$ (i.e. $k<n-1$),
 $$
 t=k.
 $$
Observe that this implies that $\G_1=\D_{t+1},\dots, \G_{n-t}=\D_n$. 
 
 Now we have
 $$
 Z(a_+)\t  Z(a_-)\h Z(\G_{n-t+1},\dots,\G_n)\t Z(\G_1,\dots,\G_{n-t}) \t Z(\D_{t+1},\dots,\D_n)\t Z(\D_1,\dots,\D_t)
 $$
 $$
 \cong Z(\G_{n-t+1},\dots,\G_{n})\t Z(\G_1,\dots,\G_{n-t},\D_{t+1},\dots,\D_n)\t Z(\D_1,\dots,\D_t)
 $$
 since $Z(\G_1,\dots,\G_{n-t}) \t Z(\D_{t+1},\dots,\D_n)$  is irreducible (we are in the essentially unitary situation, from which we get irreducibility of the induced representation).
 Now
 $$
 Z(\G_{n-t+1},\dots,\G_n)\t Z(\G_1,\dots,\G_{n-t},\D_{t+1},\dots,\D_n)\t Z(\D_1,\dots,\D_t)\h
 $$
 $$
 \zeta(\G_{n-t+1},\dots,\G_n)\t \zeta(\G_1,\dots,\G_{n-t},\D_{t+1},\dots,\D_n)\t \zeta(\D_1,\dots,\D_t)\cong \zeta(a_++a_-).
 $$
 Therefore, $Z(a_+)\t  Z(a_-)$ has a unique irreducible subrepresentation, and it is $Z(a_++a_-)$.
 
 For the quotient setting, observe that 
 $$
{\C R}(n,d)_{(-k)}^{(\rho)} \cong L(a(d,n)^{(\nu^{k/2}\rho)})\t L(a(d,n)^{(\nu^{-k/2}\rho)}).
$$
Now applying similar arguments as above (dealing with quotients instead of subrepresentations), we get that ${\C R}(n,d)_{(-k)}^{(\rho)}$ has a unique irreducible quotient, and that this quotient  is 
$$
L(a(d,n)^{(\nu^{k/2}\rho)}+a(d,n)^{(\nu^{-k/2}\rho)})=Z((a(d,n)^{(\nu^{k/2}\rho)}+a(d,n)^{(\nu^{-k/2}\rho)})^t). 
$$

To get the description of the irreducible quotient and the irreducible subrepresentation of  ${\C R}(n,d)_{(k)}^{(\rho)}$, apply the contragredient to ${\C R}(n,d)_{(-k)}^{(\tilde\rho)}$.

The proof of the proposition is now complete.
 \end{proof}
 
 \subsection{Highest derivatives of irreducible subquotients}

\begin{lemma}
\label{le-hd}
 Let $Z(b)\leq Z(a_-)\times Z(a_+)$. Denote by $c(b)$ the cardinality of $b$. Then  
$$
n\leq c(b)\leq 2n,
$$
and we get the cuspidal support of the highest derivative of $b$ from the cuspidal support of $a_-+a_+$ removing ends of all segments in $a_+$ and removing ends of the first $c(b)-n$  segments in $a_-$. Therefore, the multiset of ends of $b$ consists of all ends of segments $\G_i$, and the ends of the  first $c(b)-n$ ends of segments $\D_i$.
\end{lemma}

\begin{proof}
Recall that 
$$
\Der(Z(a_-)\t Z(a_+))= \bigg(Z(\D_1,\dots,\D_n) + Z(\D_1^-,\D_2,\dots,\D_n)+\dots+Z(\D_1^-,\dots,\D_n^-)\bigg)
$$
$$
\t \bigg( Z(\G_1,\dots,\G_n) + Z(\G_1^-,\G_2,\dots,\G_n)+\dots+Z(\G_1^-,\dots,\G_n^-)\bigg).
$$
Observe  that the end of the last segment in $a_+$ is not in the support of the highest derivative of $b$. Therefore, the highest derivative of $Z(b)$ must be a subquotient of the product, in which the second factor is $Z(\G_1^-,\dots,\G_n^-)$. Now the claim of the lemma follows directly (use \eqref{cs}).
\end{proof}

\subsection{Symmetry} 

\begin{lemma} 
\label{le-sym}
Let $Z(b)\leq Z(a_-)\times Z(a_+)$. Suppose that $\rho$ is unitarizable. Then
$$
\tau\mapsto\tau^+ 
$$
defines a bijection from the multiset of all ends of segments in $b$ onto the multiset of all beginnings of segments in $b$, i.e. 
$$
\C E(b)^+= \C B(b).
$$
\end{lemma}

\begin{proof} Observe that $Z(b^+)\leq Z(a_-)\times Z(a_+)$ since $\rho$ is unitarizable (the unitarizability implies that the representation on the right hand side is a Hermitian element of the Grothendieck group since $Z(a_-)^+=Z(a_+)$).

Denote by $c(b)$ the cardinality of $b$ (we know $n\leq c(b)\leq 2n$ from the previous proposition). Then  the
last proposition tells us that the multiset of all ends of $b$  depends only on $c(b)$. 

Obviously,  $b$ and $b^+$ have the same cardinality. 
 Therefore, the  multisets of the ends of $b^+$ and of $b$ are the same, i.e. $\C E(b)=\C E(b^+)$. Observe that one gets the multiset of the beginnings of $b$ from the multiset of the ends of $b^+$ applying $\tau\mapsto\tau^+$, i.e. $\C B(b)=\C E(b^+)^+$. Therefore $\C E(b)^+=\C E(b^+)^+= \C B(b)$, i.e. $\tau\mapsto\tau^+ 
$
is a bijection from the multiset of all ends of segments in $b$ onto the multiset of all beginnings of segments in $b$ (it preserves the multisets).
\end{proof}

\subsection{Key  lemma}

The following lemma will be crucial for exhaustion of composition series.

\begin{lemma} 
\label{le-crucial}
Let $Z(b)\leq Z(a_-)\times Z(a_+)$. Suppose that $\rho$ is unitarizable.
\begin{enumerate}

\item There exists a (multi)set $b_{\text{left}}=(\D_1',\dots,\D_n')$ such that
$$
b(\D_i')=b(\D_i),\quad i=1,\dots,n,
$$
$$
e(\D_1)\leq e(\D_1')<\dots<e(\D_n')
$$
and
that each segment of 
$b_{\text{left}}$ is contained in $b$, i.e. $b_{\text{left}}\subseteq b.$

\item There exists a (multi)set $b_{\text{right}}=(\G_1',\dots,\G_n')$ such that
$$
e(\G_i')=e(\G_i),\quad i=1,\dots,n,
$$
$$
 b(\G_1')<\dots<b(\G_n')\leq b(\G_n)
$$
and that 
each segment of 
$b_{\text{right}}$ is contained in $b$, i.e. $b_{\text{right}}\subseteq b$.

\item Suppose that $\C B(a_-)$ and $\C B(a_+)$ are disjoint and that the cardinality $c(b)$ of $b$ is strictly smaller then $2n$. Denote $l=2n-c(b)$ (i.e. $c(b)=2n-l$). Than
\begin{enumerate}
\item \quad $\D_i'=\D_i$, \quad $i=1,\dots, n-l$.

\item \quad $\G_i'=\G_i$, \quad \ $i=l+1,\dots,n$.

\item For $i=1,\dots,l$, $\D_{n-l+i}$ and $\G_{i}$ are disjoint, $\D_{n-l+i}\cup\G_{i}$ is a segment, and 
$$
\D_{n-l+i}'=\D_{n-l+i}\cup \G_{i}=\G_i'.
$$

\item \quad
$
b=(\D_1,\dots,\D_{n-l}, \D_{n-l+1}',\dots.\D_n', \G_{l+1},\dots,\G_n).
$

\item \quad $d\leq k$ (i.e. $B_-<A_+$) and $k+1\leq n+d$.

\item The multiset $b$ as above is unique. Further, $l$ is the maximal index such that $\D_n\cup \G_l\in \SC$.

\item In this situation we have
 $$
 (a_-^t+a_+^t)^t=b.
 $$
 
 \item Further,  $Z(b)$ is the unique irreducible subrepresentation   of $Z(a_-)\t Z(a_+)$ in this situation.
 
\end{enumerate}

\item Suppose that $\C B(a_-)$ and $\C B(a_+)$ are not disjoint (this is equivalent to $A_+\leq C_-$). Further, assume that $n\leq d$ (this is equivalent to $C_-\leq B_-$). Then the cardinality $c(b)$ of $b$ is equal $2n$.

\end{enumerate}
\end{lemma}

\begin{proof}  Observe that $Z(b^t)\leq Z(a_-^t)\times Z(a_+^t)=
Z(a(d,n)^{(\nu^{-k/2}\rho)})\t Z(a(d,n)^{(\nu^{k/2}\rho)})$, which implies
$$
b^t\leq a_-^t+a_+^t=
a(d,n)^{(\nu^{-k/2}\rho)}+a(d,n)^{(\nu^{k/2}\rho)}.
$$
This implies that in $b^t$ there is a unique segment $\D$ such that $b(\D)=b(\D_1)$. This segment must have obviously at least $n$ elements. Now $^\ra$MWA implies (1).
Analogously follows (2) using the dual version of MWA$^\la$ (see \cite{T-GL-red}).

Now we shall prove (3). We assume that conditions of (3) on $a_-$ and $a_+$ hold (in particular, disjointness of beginning of segments is equivalent to $C_-<A_+$). Recall that by Lemma \ref{le-hd} the multiset of ends of $b$ consists of all ends of segments $\G_i$, and the ends of the  first $n-l$ ends of segments $\D_i$. Now Lemma \ref{le-sym} implies that the multiset of beginnings of $b$ consists of all beginnings of segments $\D_i$, and the beginnings of the  last $n-l$ ends of segments $\G_i$.

Let $\rho'$ be an element of some $\D_i$ or $\G_j$. Then we can write $\rho'=\nu^x\rho$ for unique $x\in \frac 12\Z$. In this case we write
$$
\exp_\rho(\rho')=x,
$$
and we shall say that $x$ is the exponent of $\rho'$ with respect to $\rho$.

Denote by $X$ the multiset of all exponents with respect to $\rho$ of beginnings of segments in $b$. Write elements of $X$ as
$$
b_1<\dots<b_{2n-l}.
$$
Then by Lemma \ref{le-sym}, $-X$ is the multiset of all exponents with respect to $\rho$ of ends of segments in $b$. Write elements of $-X$ as
$$
e_1<\dots<e_{2n-l}.
$$
Therefore, there exists a permutation $\s$ of $\{1,\dots,2n-l\}$ satisfying
$$
b_i\leq e_{\s(i)},\quad i=1,\dots,2n-l,
$$
 such that
$$
b=([\nu^{b_1}\rho,\nu^{e_{\s(1)}}\rho], \dots ,[\nu^{b_{2n-l}}\rho,\nu^{e_{\s(2n-l)}}\rho])
$$
$$
\hskip50mm
=([\nu^{b_{\s^{-1}(1)}}\rho,\nu^{e_{1}}\rho], \dots ,[\nu^{b_{\s^{-1}(2n-l)}}\rho,\nu^{e_{2n-l}}\rho]).
$$
Now by (1) and (2) we must have
$$
\s(1)<\dots<\s(n) \quad \text{and} \quad \s^{-1}(n-l+1)<\dots <\s^{-1}(2n-l).
$$
This implies 
$$
i\leq \s(i),\quad i=1,\dots,n\qquad \text{and} \qquad \s^{-1}(j)\leq j, \quad j=n-l+1,\dots, 2n-l.
$$

To shorten the  formulas below, for $\D\in\{\D_1,\dots,\D_n,\G_1,\dots,\G_n\}$ we use the following notation:
$$
b_\rho(\D):=\exp_\rho(b(\D)) \text{\quad and \quad} e_\rho(\D):=\exp_\rho(e(\D)).
$$

Since  the   cuspidal supports of $a_-+a_+$ and $b$ must be the same, their cardinalities must be the same. The first cardinality is
$$
2nd=\sum_{i=1}^n (e_\rho(\D_i)-b_\rho(\D_i)+1)+\sum_{i=1}^n (e_\rho(\G_i)-b_\rho(\G_i)+1)
$$
From the other side, using the fact that $b_i\leq e_{\s(i)},\quad i=1,\dots,2n-l$,  the second cardinality is
$$
\sum_{i=1}^{2n-l} (e_{\s(i)}-b_i+1)=\sum_{i=1}^{2n-l} (e_{\s(i)}+1) -\sum_{i=1}^{2n-l} b_i=
\sum_{i=1}^{2n-l} (e_{i}+1) -\sum_{i=1}^{2n-l} b_i
$$
$$
=\sum_{i=1}^{n-l} (e_\rho(\D_i)-b_\rho(\D_i)+1)+\sum_{i=l+1}^n (e_\rho(\G_i)-b_\rho(\G_i)+1)
+
\sum_{i=n-l+1}^n (e_\rho(\G_{i-(n-l)})-b_\rho(\D_{i})+1).
$$ 
Since the above two cardinalities must be the same, we have
$$
\sum_{i=n-l+1}^n (e_\rho(\D_i)-b_\rho(\D_i)+1)+\sum_{i=1}^l (e_\rho(\G_i)-b_\rho(\G_i)+1)
=
\sum_{i=n-l+1}^n (e_\rho(\G_{i-(n-l)})-b_\rho(\D_{i})+1).
$$
This further implies 
$$
\sum_{i=n-l+1}^n e_\rho(\D_i)+\sum_{i=1}^l (e_\rho(\G_i)-b_\rho(\G_i)+1)
=
\sum_{i=n-l+1}^n e_\rho(\G_{i-(n-l)})
=
\sum_{i=1}^l e_\rho(\G_{i}).
$$
Thus 
$$
\sum_{i=n-l+1}^n e_\rho(\D_i)+\sum_{i=1}^l (-b_\rho(\G_i)+1)
=
0,
$$
i.e.
$$
\sum_{i=1}^l e_\rho(\D_{n-l+i})+\sum_{i=1}^l (-b_\rho(\G_i)+1)
=
\sum_{i=1}^l (e_\rho(\D_{n-l+i})-b_\rho(\G_i)+1)
=
0.
$$
Since all the numbers 
$e_\rho(\D_{n-l+i})-b_\rho(\G_i)+1$, $i=1,\dots,l$ are the same, we get that
\begin{equation}
\label{contact}
e_\rho(\D_{n-l+i})+1=b_\rho(\G_i), \quad i=1,\dots,l.
\end{equation}
For $i=l$ we get $e_\rho(\D_n)+1=b_\rho(\G_l)$, i.e. $D_-+1=A_++l-1$. Going to $n,d,k$-notation, we get $\frac{d-1}2+\frac{n-1}2-\frac k2+1=- \frac{d-1}2-\frac{n-1}2+\frac k2+l-1$, which gives
 $d+n= k +l$. This implies
$$
l=n+d-k.
$$
Recall $1\leq l\leq n$. Therefore
$$
k+1\leq n+d, 
$$
$$
d\leq k.
$$
The last inequality is equivalent to $B_-<A_+$.

We illustrate the situation by the following example:
\begin{equation} 
\label{eq-ex-1}
\xymatrix@C=.6pc@R=.1pc
{ 
 \bullet& \bullet\ar @{-}[l] & \bullet\ar @{-}[l] & \bullet\ar @{-}[l] & \bullet\ar @{-}[l]  
\\
&\bullet& \bullet\ar @{-}[l] & \bullet\ar @{-}[l] & \bullet\ar @{-}[l] & \bullet\ar @{-}[l]  
&
\circ& \circ\ar @{-}[l] & \circ\ar @{-}[l] & \circ\ar @{-}[l] & \circ\ar @{-}[l] 
\\
&&\bullet& \bullet\ar @{-}[l] & \bullet\ar @{-}[l] & \bullet\ar @{-}[l] & \bullet\ar @{-}[l]  
&
\circ& \circ\ar @{-}[l] & \circ\ar @{-}[l] & \circ\ar @{-}[l] & \circ\ar @{-}[l] 
\\
&&&&&&&&\circ& \circ\ar @{-}[l] & \circ\ar @{-}[l] & \circ\ar @{-}[l] & \circ\ar @{-}[l]  
} 
\end{equation}

Our next aim will be to prove that the permutation $\s$ is the identity permutation.

 In the following considerations, we shall use the fact that the cuspidal supports of $b$ and $a_-+a_+$ must coincide (in the previous considerations we have  used only the fact that their cardinalities   must be the same).
 
Consider the case $B_-+1<A_+.$ Suppose $\s(1)\ne1$. Then $1<\s(1)$. Recall $i\leq \s(i)$, $1\leq i\leq n$. This implies that the multiplicity of $\nu^{B_-+1}\rho$ in the cuspidal support of $b$ is strictly bigger then the multiplicity in the cuspidal support of $a_-+a_+$. Thus, $\s(1)=1$. In the same way we get $\s(2)=2$ if $B_-+2<A_+$. Continuing in this way,  we get that
$$
\s(i)=i, \quad \text{if} \quad B_-+i<A_+.
$$
We shall now find the maximal $i$ such that $B_-+i<A_+$. In $n,d,k$-notation, this condition becomes $\frac{d-1}2-\frac{n-1}2-\frac k2+i<- \frac{d-1}2-\frac{n-1}2+\frac k2$, i.e.
$d-1+i<k$, which is equivalent to $d+i\leq k$. Therefore
\begin{equation}
\label{eq-2-1}
\s(i)=i \quad \text{for all}\quad 1\leq i \leq k-d=n-l,
\end{equation}
since we know $l=n+d-k$.

This implies that $\s$ carries $\{n-l+1,\dots,2n-l\}$ into itself. Since $\s^{-1}$ is monotone on this set, we get that $\s$ is the identity permutation. This completes the proof of (a) - (d) in (3). 

It remains to prove (g). Because of \eqref{t}, it is enough to prove
$$
(a(d,n)^{(\nu^{-k/2}\rho)}+a(d,n)^{(\nu^{k/2}\rho)})^t=b.
$$
The proof of this relation is a simple straight-forward application of MWA$^\la$, and we shall not present in detail here. Rather we shall illustrate the proof with an example \eqref{eq-ex-1}. We illustrate $a_-^t+a_+^t$ by the drawing
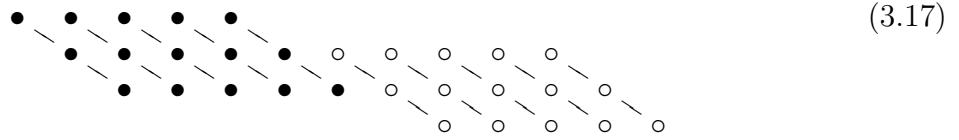
\begin{equation} 
\label{eq-ex-2}
\xymatrix@C=.6pc@R=.1pc
{ 
 \bullet& \bullet & \bullet& \bullet& \bullet  
\\
&\bullet\ar @{-}[lu] & \bullet\ar @{-}[lu] & \bullet\ar @{-}[lu] & \bullet\ar @{-}[lu] & \bullet\ar @{-}[ul]  
&
\circ& \circ& \circ & \circ & \circ
\\
&&\bullet\ar @{-}[ul] & \bullet\ar @{-}[ul] & \bullet\ar @{-}[lu] & \bullet\ar @{-}[lu] & \bullet\ar @{-}[ul]  
&
\circ\ar @{-}[lu] & \circ\ar @{-}[ul] & \circ\ar @{-}[ul] & \circ\ar @{-}[lu] & \circ\ar @{-}[ul] 
\\
&&&&&&&&\circ\ar @{-}[ul] & \circ\ar @{-}[ul] & \circ\ar @{-}[ul] & \circ\ar @{-}[ul] & \circ\ar @{-}[lu]  
} 
\end{equation}
Now MWA$^\la$ applied to $a_-^t+a_+^t$ can be illustrated by:
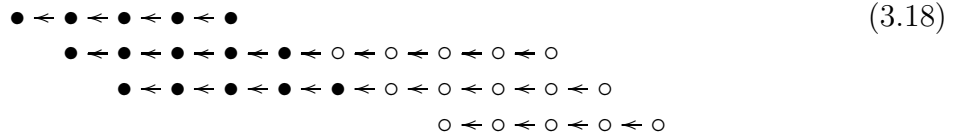
\begin{equation} 
\label{eq-ex-3}
\xymatrix@C=.6pc@R=.1pc
{ 
 \bullet& \bullet\ar [l] & \bullet\ar[l] & \bullet\ar [l] & \bullet\ar [l]  
\\
&\bullet& \bullet\ar [l] & \bullet\ar [l] & \bullet\ar [l] & \bullet\ar [l]  
&
\circ\ar[l]& \circ\ar [l] & \circ\ar [l] & \circ\ar [l] & \circ\ar [l] 
\\
&&\bullet& \bullet\ar [l] & \bullet\ar [l] & \bullet\ar [l] & \bullet\ar [l]  
&
\circ\ar[l]& \circ\ar [l] & \circ\ar [l] & \circ\ar [l] & \circ\ar [l] 
\\
&&&&&&&&\circ& \circ\ar [l] & \circ\ar [l] & \circ\ar [l] & \circ\ar [l]  
} 
\end{equation}
The general proof goes in the same way.

The claim (h) follows from Proposition \ref{prop-sub-quo} and (g) (we  use this in the proof of Theorem \ref{th-diff}).

It remains to  prove (4). The proof is a simple modification of the proof of (3). We assume that conditions of (3) on $a_-$ and $a_+$ hold (in particular, $A_+\leq C_-$, which means that the beginnings of segments are not disjoint). 
Suppose $c(b)<n$. Write $c(b)=2n-l$. Then
$$
1\leq l\leq n.
$$
Write 
$$
b=(\D_1',\dots,\D_{2n-l}').
$$

Since  the cuspidal supports of $a_-+a_+$ and $b$ must be the same, their cardinalities must be the same. The first cardinality is
$$
2nd=\sum_{i=1}^n (e_\rho(\D_i)-b_\rho(\D_i)+1)+\sum_{i=1}^n (e_\rho(\G_i)-b_\rho(\G_i)+1)
$$
From the other side, using the fact that $b_\rho(\D_i')\leq e_\rho(\D_i'),\quad i=1,\dots,2n-l$,  the second cardinality is
\begin{equation}
\label{e}
\sum_{i=1}^{2n-l} (e_\rho(\D_i') - b_\rho(\D_i')+1)=\sum_{i=1}^{2n-l} (e_\rho(\D_i') +1)-\sum_{i=1}^{2n-l}b_\rho(\D_i').
\end{equation}
The multiset of ends of $b$ consists of all ends of segments $\G_i$, and the ends of the  first $n-l$ ends of segments $\D_i$. Further, the multiset of beginnings of $b$ consists of all beginnings of segments $\D_i$, and the beginnings of the  last $n-l$ ends of segments $\G_i$. Using this fact, and permuting elements in the sums, we easily get that we can write \eqref{e} as
$$
\sum_{i=1}^{n-l} (e_\rho(\D_i)-b_\rho(\D_i)+1)+\sum_{i=l+1}^n (e_\rho(\G_i)-b_\rho(\G_i)+1)
+
\sum_{i=n-l+1}^n (e_\rho(\G_{i-(n-l)})-b_\rho(\D_{i})+1).
$$ 
The above two cardinalities must be the same, which implies
$$
\sum_{i=n-l+1}^n (e_\rho(\D_i)-b_\rho(\D_i)+1)+\sum_{i=1}^l (e_\rho(\G_i)-b_\rho(\G_i)+1)
=
\sum_{i=n-l+1}^n (e_\rho(\G_{i-(n-l)})-b_\rho(\D_{i})+1).
$$
This further implies 
$$
\sum_{i=n-l+1}^n e_\rho(\D_i)+\sum_{i=1}^l (e_\rho(\G_i)-b_\rho(\G_i)+1)
=
\sum_{i=n-l+1}^n e_\rho(\G_{i-(n-l)})
=
\sum_{i=1}^l e_\rho(\G_{i}).
$$
Thus 
$$
\sum_{i=n-l+1}^n e(\D_i)+\sum_{i=1}^l (-b_\rho(\G_i)+1)
=
0,
$$
i.e.
$$
\sum_{i=1}^l e_\rho(\D_{n-l+i})+\sum_{i=1}^l (-b_\rho(\G_i)+1)
=
\sum_{i=1}^l (e_\rho(\D_{n-l+i})-b_\rho(\G_i)+1)
=
0.
$$
Since all the numbers 
$e_\rho(\D_{n-l+i})-b_\rho(\G_i)+1$, $i=1,\dots,l$, are the same, we get that
\begin{equation}
\label{contact+}
e_\rho(\D_{n-l+i})+1=b_\rho(\G_i), \quad i=1,\dots,l.
\end{equation}
For $i=l$ we get $e_\rho(\D_n)+1=b_\rho(\G_l)$, i.e. $D_-+1=A_++l-1$. Going to $n,d,k$-notation, we get $\frac{d-1}2+\frac{n-1}2-\frac k2+1=- \frac{d-1}2-\frac{n-1}2+\frac k2+l-1$, which gives
 $d+n= k +l$. This implies
$$
l=n+d-k.
$$
Recall $1\leq l\leq n$. Therefore
$$
k+1\leq n+d, 
$$
$$
d\leq k.
$$
The last inequality is equivalent to $B_-<A_+$.

From the other side,  we suppose $A_+\leq C_-$ and $C_-\leq B_-$, which implies $A_+\leq B_-$. This contradicts to $B_-<A_+$.

The proof is now complete.
\end{proof}

\section{Definition of multisegments representing composition series}

Denote
$$
r_0(n,d)^{(\rho)}_k=(\D_{1},\dots,\D_n,\G_1,\dots,\G_n).
$$
For $1\leq j\leq n$ we define multisegments
$$
r_j(n,d)^{(\rho)}_k
$$
whenever $\D_n\ra \G_j$. Observe that $\D_{n}\ra \G_{j}$ 
if and only if $\D_{n-j+1}\ra \G_1$. This is the case
 if and only if 
 $$
 1\leq b_\rho(\G_1)- b_\rho(\D_{n-j+1})\leq d.
 $$
The last  condition becomes 
$1\leq (-\frac{d-1}2-\frac{n-1}2+\frac k2) -( -\frac{d-1}2-\frac{n-1}2-\frac k2+n-j)\leq d$. Therefore if $j\ne0$, the  multisegments $r_j(n,d)^{(\rho)}_k$ are defined for indexes
$$
\max( n-k+1,1)\leq j \leq \min(n-k+d,n).
$$

Then we get $r_j(n,d)^{(\rho)}_k$ by replacing
in
$$
r_0(n,d)^{(\rho)}_k=(\D_{1},\dots,\D_n,\G_1,\dots,\G_n)
$$
the part
$$
\D_{n-j+1},\dots,
\D_n,\G_1,
\dots,\G_j
$$
with
$$
\D_{n-j+1}\cup \G_1,\D_{n-j+1}\cap \G_1,
\dots\dots,
\D_{n}\cup\G_j,\D_n\cap\G_j
$$
(we omit  $\emptyset$ if it shows up in the above formula and the formulas below).

In other words we have
$$
r_j(n,d)_{(k)}^{(\rho)}=\sum_{i=1}^j(\D_{i+n-j}\cup \G_i, \D_{i+n-j}\cap \G_i)+\sum_{i=1}^{n-j}(\D_i,\G_{n-i+1})
$$
$$
\hskip22mm=\sum_{i=1}^j(\D_{i+n-j}\cup \G_i, \D_{i+n-j}\cap \G_i)+\sum_{i=1}^{n-j}(\D_i,\G_{i+j})
$$
for $j=0$ and for  $1\leq j\leq n$ for which $\D_n\ra \G_j$.

Suppose that $\D$ and $\G$ are segments such that their intersection is non-empty. Then obviously
 $$
 ((\D\cup\G)^-,(\D\cap\G)^-)=(\D^-\cup\G^-,\D^-\cap\G^-)
 $$
(as above, we omit  $\emptyset$ if it shows up). This implies that 
\begin{equation}
\label{der-r}
\text{h.d.}(Z(r_j(n,d)_{(k)}^{(\rho)}))=\nu^{-1/2}Z(r_j(n,d-1)_{(k)}^{(\rho)})
\end{equation}
if  $d\geq2$ and $\D_n\ra\G_j$ such that $\D_n\cap\G_j\ne \emptyset$.

\section{Composition series - disjoint beginnings of segments}

We continue with the notation of the previous section. In this section we shall assume that
$$
C_-<A_+,
$$
i.e. $\{b(\D_1),\dots,b(\D_n)\}\cap \{b(\G_1),\dots,b(\G_n)\}=\emptyset.$ We also assume
$$
A_+\leq D_-+1
$$
(recall that for $A_+> D_-+1$ the representation ${\C R}(n,d)_{(k)}^{(\rho)}$ is irreducible). 

In the $n,d,k$-notation these two conditions become
$-\frac{d-1}2+\frac{n-1}2-\frac k2 < -\frac{d-1}2-\frac{n-1}2+\frac k2$ and 
$-\frac{d-1}2-\frac{n-1}2+\frac k2 \leq \frac{d-1}2+\frac{n-1}2-\frac k2+1$, i.e.
$$
n\leq k,
$$
$$
k\leq n+d-1
$$
(i.e. $k<n+d$).

Let fix $j$ such that $\D_n\ra\G_j$. 
This is the case if and only if 
 $$
 A_++j-1\leq D_-+1.
 $$
  The last condition becomes 
$-\frac{d-1}2-\frac{n-1}2+\frac k2+j-1 \leq \frac{d-1}2+\frac{n-1}2-\frac k2+1$, i.e.  $k+j\leq d+n$. Therefore, the multisegments $r_j(n,d)_{(k)}^{(\rho)}$ are defined for indexes
$$
0\leq j \leq \min(n+d-k,n)
$$
in the case that we consider in this section.

Observe that in the case $k\leq d$, $r_j(n,d)_{(k)}^{(\rho)}$ are defined for all $0\leq j\leq n$, and then we have
$$
r_0(n,d)_{(k)}^{(\rho)}=\sum_{i=1}^{n}(\D_i,\G_{n-i+1}).
$$
$$
r_1(n,d)_{(k)}^{(\rho)}=(\D_n\cup \G_1, \D_n\cap \G_1)+\sum_{i=1}^{n-1}(\D_i,\G_{i+1})
$$
\hskip40mm \dots \dots \dots \dots \dots \dots \dots \dots \dots \dots \dots \dots \dots \dots \dots \dots 
$$
r_j(n,d)_{(k)}^{(\rho)}=\sum_{i=1}^j(\D_{i+n-j}\cup \G_i, \D_{i+n-j}\cap \G_i)+\sum_{i=1}^{n-j}(\D_i,\G_{i+j}).
$$
\hskip40mm \dots \dots \dots \dots \dots \dots \dots \dots \dots \dots \dots \dots \dots \dots \dots \dots
$$
r_n(n,d)_{(k)}^{(\rho)}=\sum_{i=1}^n(\D_i\cup \G_i, \D_i\cap \G_i).
$$
In the case $d<k$, let $j$ is the greatest index for which $\D_n\cup \G_j$ is still a segment (with the assumptions of the present section, this implies $\D_n\ra \G_j$).
Then only multisegments $r_0(n,d)_{(k)}^{(\rho)}, \dots, r_j(n,d)_{(k)}^{(\rho)}$ are defined (i.e. the first $j+1$ terms above are defined).

\begin{proposition}
\label{prop-diff}
 Let $A_+\leq D_-+2$, i.e.
$$
k\leq n+d
$$
(otherwise, ${\C R}(n,d)_{(k)}^{(\rho)}$
 is irreducible\footnote{Observe that also for $k=n+d$ we have irreducibility.}). Suppose $C_-<A_+$, i.e.
$$
n\leq k.
$$
Then

\begin{enumerate} \item ${\C R}(n,d)_{(k)}^{(\rho)}$ is a multiplicity one representation.

\item  Let $B_-<A_+$. In that case
$$
d\leq k\leq n+d.
$$
Then ${\C R}(n,d)_{(k)}^{(\rho)}$ has length $n+d+1-k$, and its composition series consists of 
$$
Z(r_i(n,d)_{(k)}^{(\rho)}), \quad 0\leq i \leq n+d-k.
$$

\item Suppose $A_+\leq B_-$, i.e.
$$
k <d    .
$$
Then ${\C R}(n,d)_{(k)}^{(\rho)}$ has length $n+1$, and its composition series consists of 
$$
Z(r_i(n,d)_{(k)}^{(\rho)}), \quad 0\leq i \leq n.
$$
\end{enumerate}
\end{proposition}

\begin{proof} One can easily see that if any of the claims of the above proposition holds for a unitarizable $
\rho$, then it holds for any $\nu^\a\rho, \a \in \mathbb R$ (for that $n,d$ and $k$). This follows from the fact that in general holds $\nu^\a Z(a)\cong Z(\nu^\a a)$, $\nu^\a(\pi_1\t\pi_2)\cong (\nu^\a \pi_1)\t(\nu^\a\pi_2)$, $\nu^\a({\C R}(n,d)_{(k)}^{(\rho)})\cong {\C R}(n,d)_{(k)}^{(\nu^\a\rho)}$ and $\nu^\a(r_i(n,d)_{(k)}^{(\rho)})\cong r_i(n,d)_{(k)}^{(\nu^\a\rho)}$.

We shall first prove (2). While proving (2), we shall prove also that all the multiplicities are one.

We fix  $n$ and $k$. The proof will go by induction with respect to $d$. Observe that $k-n\leq d \leq k$.
For $d=k-n$ (i.e. $k=n+d$), we know that ${\C R}(n,d)_{(k)}^{(\rho)}$ is irreducible, and isomorphic to $Z(r_0(n,d)_{(k)}^{(\rho)})$. Therefore, we have the basis of the induction.

Suppose $d=1$. Then $k\leq n+1$ and $n\leq k$. If $k=n+1$, we have observed above that (2) holds. Let $k=n$. Then easily follows from \cite{Z} that (2), together with multiplicity one, holds. Now in the rest of the proof, it is enough to consider the case of $d>1$.

 Let
$$
k-n< d\leq k,
$$
and suppose that (2) holds for  $d-1$, together with the multiplicity one claim (in (1)). 
We shall now show that (2) holds also for $d$ (together with the multiplicity one claim). 
Recall
\begin{equation}
\label{hd-r}
\text{h.d.}({\C R}(n,d)_{(k)}^{(\rho)})=\nu^{-1/2}{\C R}(n,d-1)_{(k)}^{(\rho)}
\end{equation}
for $d\geq 2$.
Also
\begin{equation}
\label{hd-z}
\text{h.d.}(Z(r_j(n,d)_{(k)}^{(\rho)}))=\nu^{-1/2}Z(r_j(n,d-1)_{(k)}^{(\rho)})
\end{equation}
for  
$$
0\leq j\leq n+d-k-1.
$$
From the first relation and the inductive assumption follows that ${\C R}(n,d)_{(k)}^{(\rho)}$ has precisely $n+d-k$ irreducible subquotients $\pi=Z(a)$ such that the cardinality of $a$ is $2n$. It also implies that all these subquotients have multiplicity one.
The inductive assumption tells us that the irreducible subquotients of the highest  derivative  are $\nu^{-1/2}Z(r_j(n,d-1)_{(k)}^{(\rho)})$, $0\leq j \leq n+(d-1)-k$.
 Now the second relation above implies that these subquotients $Z(a)$ are representations $Z(r_j(n,d)_{(k)}^{(\rho)})$, $0\leq j \leq n+(d-1)-k$. Now Proposition \ref{prop-sub-quo} and Lemma \ref{le-crucial} (3) (g) imply also that $Z(r_{n+d-k}(n,d)_{(k)}^{(\rho)})$ is irreducible subquotient, and that the multiplicity of this subquotient  is one. It remains to prove that these are all the irreducible subquotients.
Let $\pi=Z(a)$ be an arbitrary irreducible subquotient of ${\C R}(n,d)_{(k)}^{(\rho)}$. If the cardinality of $a$ is $2n$, then we have seen that it must be one of $Z(r_j(n,d)_{(k)}^{(\rho)})$, $0\leq j \leq n+(d-1)-k$. Suppose that the cardinality of $a$ is strictly smaller then $2n$. 
Now we shall find an index $l$ such that $e_\rho(\D_n)+1=b_\rho(\G_l)$, i.e. $\frac{n-1}2+\frac{d-1}2-\frac{k}2 +1=-\frac{n-1}2-\frac{d-1}2+\frac{k}2 +l -1$, which implies
$$
\ell =n+d-k.
$$
Then (3) of Lemma \ref{le-crucial} implies that $\pi\cong Z(r_{n+d-k}(n,d)_{(k)}^{(\rho)})$. This completes the proof of (2).

Now we shall prove (3).  At the same time we shall prove also that all the multiplicities are one. This and the first part of the proof, will then give (1).

We  prove (3) by induction with respect to $d$. Observe that we need to prove for $k<d$. Actually, we shall prove the claim of (3) (together with multiplicity one property) by induction for
$$
k\leq d.
$$
For $d=k$, claim (3) holds by (2) (observe that in this case $n+d-k$ is $n$, and we have $n+1$ irreducible subquotients). Also the multiplicity one holds in this case.

Let $k<d$, and suppose that our claim holds for $d-1$. Then in the same way as in the first part of the proof, looking at the highest derivative of ${\C R}(n,d)_{(k)}^{(\rho)}$, the inductive assumption implies that $Z(r_j(n,d)_{(k)}^{(\rho)})$, $0\leq j \leq n$ are subquotients, and all have multiplicity one. Further, (4) of Lemma \ref{le-crucial} and the condition $A_+\leq B_-$ (which we assume in (3)) imply that these are all the irreducible subquotients.

The proof of the proposition is now complete.
\end{proof}

Summing up, we get the following:

\begin{theorem}
\label{th-diff}
 Let $k\in\Z_{\geq 0}$. 
\begin{enumerate}

\item The representations   ${\C R}(n,d)_{(k)}^{(\rho)}$ and   ${\C R}(n,d)_{(-k)}^{(\rho)}$ have the same composition series.
\item
For 
$$
 n+d\leq k
 ,
 $$
   ${\C R}(n,d)_{(k)}^{(\rho)}$ is irreducible, and ${\C R}(n,d)_{(k)}^{(\rho)}\cong Z(r_0(n,d)_{(k)}^{(\rho)})$.

\item
 For 
 $$
n\leq k\leq n+d,
$$
 ${\C R}(n,d)_{(k)}^{(\rho)}$ is a multiplicity one representation. Its composition series consists of
$$
Z(r_i(n,d)_{(k)}^{(\rho)}), \quad 0\leq i \leq \min(n,n+d-k). 
$$

\item The representation ${\C R}(n,d)_{(k)}^{(\rho)}$ has a unique irreducible subrepresentation as well as a unique irreducible quotient. The irreducible quotient  is isomorphic to $Z(r_0(n,d)_{(k)}^{(\rho)})$, and the irreducible subrepresentation  is isomorphic to $Z(r_{\min(n,n+d-k)}(n,d)_{(k)}^{(\rho)}).$ 

For ${\C R}(n,d)_{(-k)}^{(\rho)}$, we have opposite situation regarding irreducible subrepresentation and quotient.
\qed
\end{enumerate}
\end{theorem}

\begin{proof} It remains to prove only  the claim regarding the irreducible quotient in (4). For $d\leq k \leq d+n-1$, this follows from (g) in (3) of Lemma \ref{le-crucial}. Therefore, we shall suppose 
$$
k<d.
$$
 For this, one needs to prove $(r_{0}(d,n)_{(k)}^{(\rho)})^t=r_{n}(n,d)_{(k)}^{(\rho)}$. One can get this directly by MWA$^\la$ and Proposition \ref{prop-sub-quo}. We shall give here a different argument. Observe that 
 $$
 r_{0}(d,n)_{(k)}^{(\rho)}=a(d+k,n)^{(\rho)}+a(d-k,n)^{(\rho)}.
 $$
 Now we have
 $$
 Z(a(d+k,n)^{(\rho)}+a(d-k,n)^{(\rho)})^t=Z(a(d+k,n)^{(\rho)})^t\t Z(a(d-k,n)^{(\rho)})^t= 
 $$
 $$
 Z(a(n,d+k)^{(\rho)})\t Z(a(n,d-k)^{(\rho)})=Z(a(n,d+k)^{(\rho)}+a(n,d-k)^{(\rho)}),
 $$
 since  the unitary parabolic induction is irreducible for general linear groups. This implies that 
 $$
 (r_{0}(d,n)_{(k)}^{(\rho)})^t= a(n,d+k)^{(\rho)}+a(n,d-k)^{(\rho)}.
 $$
 One gets directly that
 $$
 r_{n}(n,d)_{(k)}^{(\rho)}= a(n,d+k)^{(\rho)}+a(n,d-k)^{(\rho)}.
 $$
 This completes the proof that $(r_{0}(d,n)_{(k)}^{(\rho)})^t=r_{n}(n,d)_{(k)}^{(\rho)}$.
 
 Considering the Hermitian contragredients, we get that ${\C R}(n,d)_{(-k)}^{(\rho)}$ has opposite irreducible quotient and subrepresentation from their position in ${\C R}(n,d)_{(k)}^{(\rho)}$.
 
 The proof is now complete.
\end{proof}

\section{Derivatives and the lattice of subrepresentations - disjoint beginnings of segments}
\label{sec-der-disjoint-b}

\subsection{On the lattice of subrepresentations of a multiplicity one representation of finite length}
We shall first present several simple and well known observations about the lattice of the subrepresentations of a  multiplicity free representation $(\pi,V)\in Alg_{f.l.}(G_n)$ (actually, the general discussion holds on the level of modules).

Let $\pi$ be as above (i.e. a multiplicity one representation of finite length), and denote by
$$
J.H.(\pi)
$$
the set of all irreducible subquotients of $\pi. $

Fix any $\s\in J.H.(\pi)$, and let $V_1$ and $V_2$ be two subrepresentations of $V$ such that both of them have $\s$ for a subquotient. The multiplicity one implies that their intersection has also $\s$
for subquotient. Define
$$
\pi_\text{smallest}(\s)
$$
to be the smallest (with respect to the inclusion) subrepresentation of $\pi$ which has $\s$ for a subquotient (it is equal to the intersection of all subrepresentations of $ \pi$ which have $\s$ for  a subquotient.

Now one directly sees that for a multiplicity one representation of finite length $\pi$ and $\s,\s_1,\s_2\in J.H.(\pi)$ holds the following:

\begin{enumerate}
\item
$\pi_\text{smallest}(\s)$ is a cyclic representation.

\item $\pi_\text{smallest}(\s)$ has a minimal length among all the subrepresentations of $\pi$ which have $\s$ for a sub quotient (this property characterizes $\pi_\text{smallest}(\s)$).

\item
 $\pi_\text{smallest}(\s)$ has a unique irreducible quotient, which is  $\s$ (this property also characterizes $\pi_\text{smallest}(\s)$). 
 
%\item If we have a subrepresentation of $\pi$ which satisfies the that it has a unique irreducible quotient, which is $\s$, then it must be 
% $\pi_\text{smallest}(\s)$.

\item
 $\pi_\text{smallest}(\s_1) =\pi_\text{smallest}(\s_2) \iff \s_1=\s_2$.

\item The map
$$
V\mapsto J.H.(V)
$$
is an injective map from the set of all subrepresentations of $V$ into the partitive  set $\mathcal P(J.H.(\pi))$ of $J.H.(\pi)$.

\item
For two subrepresentations $V_1$ and $V_2$ of $V$ holds
$$
V_1\subseteq V_2 \iff J.H.(V_1)\subseteq J.H.(V_2).
$$

\end{enumerate}

\begin{definition} We shall say that $\pi$ has a minimal lattice of subrepresentations if 
$$
\pi_\text{\em smallest}(\s), \s \in J.H.(\pi)
$$
is a complete lattice of the non-zero subrepresentations of $\pi$
\end{definition}

Let now $\pi$ be a multiplicity one representation of  finite length  $n$. Suppose that $\pi$ has  a minimal  lattice of subrepresentations. Let $\s_1,\s_2$ be different members of $J.H.(\pi)$. Then  one proves easily

\begin{enumerate}

\item[(i)]   
$
\text{Either \ \ } \s_1\in J.H.(\pi_\text{smallest}(\s_2)) \text{\ \ or\ \ }\s_2 \in J.H. (\pi_\text{smallest}(\s_1))
$ 

\item[(ii)]   
$
\text{Either \ \ } \pi_\text{smallest}(\s_1) \subseteq\pi_\text{smallest}(\s_2) \text{\ \ or\ \ }\pi_\text{smallest}(\s_2) \subseteq \pi_\text{smallest}(\s_1)
$
(this property characterizes  representations with a minimal  lattice of subrepresentations).

\item[(iii)]   There exists an enumeration
$$
\s_1,\s_2,\dots,\s_n
$$
of the members of $J.H.(\pi)$, such that the map
$$
V' \mt J.H.(V')
$$
is an isomorphism of the lattice of subrepresentations of $\pi$ and  %bijection which preserves inclusions from the set of all non-zero subrepresentations of $\pi$ onto the set
%such that
$$
\{\{0\},\{\s_1\},\{\s_1,\s_2\},\dots,\{\s_1,\s_2,\dots,\s_n\}\}.
$$
The above enumeration is uniquely determined by the requirement $i<j \iff \s_i\in J.H.(\pi_\text{smallest}(\s_j))$.

\item[(iv)] The cardinal number of  the set of all non-zero subrepresentations of $\pi$ is $n$, i.e. the length of $\pi$
%\footnote{Actually, this property is equivalent to the fact that $\pi$ is a representation with a minimal lattice of subrepresentations.}.
(this property characterizes  representations with a minimal  lattice of subrepresentations).

\end{enumerate}

Therefore, a multiplicity one representations of finite length with a minimal   lattice of subrepresentations has a very simple  lattice of  subrepresentations.  To determine explicitly this lattice, one needs only to determine which of the inclusions  in (i) hold.

\subsection{Derivatives of representations}

Now we shall recall of the definition of the derivatives of representations. 
We shall 
follow the notation of the third section of \cite{BZ}. 
The group $G_{n-1}$ is imbedded into $G_n$ in a usual way: $g\mapsto 
\left(
\begin{matrix}
g&0\\0&1
\end{matrix}
\right).
$
The subgroup of $G_n$ consisting  of all the matrices which have the bottom raw equal to $(0,\dots,0,1)$
is denoted by $P_n$. Its unipotent radical is denoted by $V_n$.

Let $(\pi,W)$ be a smooth representation of $G_n$. Denote by
$$
\Psi^-(\pi)
$$ 
the normalized Jacquet module of $\pi$, i.e. $W/
W(V_n,1_{V_n})$, where 
$$
W(V_n,1_{V_n})=
\text{span}_{\mathbb C}\{\pi(v)w-w;v\in V_n,w\in W\}
$$ 
(the action of $G_{n-1}$ is the quotient action, twisted by the square root of the modular character of $P_n$; see \cite{BZ}). One  extends $\Psi^-$ to a functor
$$
\Psi^-:Alg(P_n) \ra Alg(G_{n-1})
$$ 
in a standard way.

Observe that
$
\Psi^-:W\mapsto W/W(V_n,1_{V_n})
$
carries invariant subspaces in one category to the invariant subspaces in the other one (although $\Psi^-$
is not an intertwining of representations of $G_{n-1}$).

We fix a non-trivial character $\psi$ of the additive group of the field,
and define a character $\theta$ of $V_n$ by $v\mapsto \psi(v_{n-1,n}).$  This character is normalized by $P_{n-1}$. Let now $(\pi,W)$ be a smooth representation of $P_n$. Denote $W(V_n,\theta)=
\text{span}_{\mathbb C}\{\pi(v)w-\theta(v)w;v\in V_n,w\in W\}$. Define now $\Phi^-(\pi)$ to be the quotient representation of $P_{n-1}$ on $W/W(V_n,\theta)$, again twisted by a square root of the modular character (see \cite{BZ} for details).
One extends $\Phi^-$ to a functor 
$$
\Phi^-:Alg(P_n)\ra Alg(P_{n-1}).
$$
Again this functor  carries invariant subspaces in one category to the invariant subspaces in the other category (although it is not intertwining of objects).  

Both above functors are exact.
We have natural functor $Alg(G_n)\ra Alg(P_n)$.

Now for $1\leq k \leq n$  consider the functor
$$
\Psi^-\circ (\Phi^-)^{k-1}:Alg(P_n)\ra Alg(G_{n-k}).
$$
Actually, we shall consider this functor only on $Alg(G_n)$.
This functor is called $k$-th derivative, and it is denoted by
$$
\pi \mapsto \pi^{(k)}.
$$
One takes $\pi^{(0)}$ to be just $\pi$.
Observe that again the $k$-th derivative functor carries invariant subspaces to invariant subspaces. One defines the highest derivative of a representation in the same way as we did in section 3. 

Lemma 4.5 of \cite{BZ} tells us that $(\pi_1\t\pi_2)^{(k)}$ is glued from the representations 
\begin{equation}
\label{prod-der}
\pi_1^{(i)}\t\pi_2^{(k-i)}, \ \ 0\leq i \leq k.
\end{equation}
In other words, there exists filtration $\{0\}=U_0\subseteq U_1\subseteq \dots \subseteq U_k$ of the representation space of $(\pi_1\t \pi_2)^{(k)}$, such that the sequence of representations $U_i/U_{i-1}$ coincide (up to isomorphisms) to the sequence \eqref{prod-der}, after a suitable renumeration.

\begin{remark}
\label{rm-der-rep}
Observe that in  the case of the highest derivatives, 
the above result  gives as the complete description of the highest derivatives of the product of representations. It implies
the following. Let $\pi_i^{(k_i)}$ be the highest derivatives of $\pi_i$, for $i=1,2$. Then  the highest derivative of the parabolically  induced representation
$\pi_1\t\pi_2$ is $(\pi_1\t\pi_2)^{(k_1+k_2)}$, and we have
$$
(\pi_1\t\pi_2)^{(k_1+k_2)}\cong \pi_1^{(k_1)}\t\pi_2^{(k_2)}.
$$
\end{remark}

\subsection{Connection with the derivatives on the level of of $R$}

 If $\rho$ is an irreducible cuspidal representation of $G_n$, then $\rho^{(i)}=0$ for $0<i<n$ and $\rho^{(0)}=Z(\emptyset)$ (Theorem 4.4 of \cite{BZ}). From this and \eqref{prod-der} easily follows that $\pi^{(i)}$ is a finite length representation, for any $0\leq i
 \leq n$ and any $\pi\in \tilde G_n$. Therefore, one can define
 $$
 \Der'(\pi)=\sum_{i=0}^n
\text{s.s.}(\pi^{(i)})\in R.
 $$
 Extend $\Der'$ to an additive endomorphism of $R$. Obviously, $\Der'$ is positive. Then \eqref{prod-der} implies that the additive endomorphism $\Der'$ is actually   a ring homomorphism. 
 
 Theorem 8.1 of \cite{Z} tells us that the highest derivative of a representation $Z(a)$ is $Z(a^-)$, for $a \in M(\SC)$. This implies
 $$
 \Der=\Der'.
 $$
 Therefore, $\Der$ is positive and the highest derivative of $Z(a)$ on the level of $R$ is $Z(a^-)$ (as we noted in 2.10).

\subsection{Lattice of subrepresentations}

\begin{proposition}
\label{prop-lat-diff}
 Let $k\in\Z_{\geq 0}$
and
 $
n\leq k\leq n+d.
$ Denote $\ell=\min(n,n+d-k)$.
Then ${\C R}(n,d)_{(k)}^{(\rho)}$ is a multiplicity one representation of length $\ell+1$, and
it is a representation with the minimal lattice of subrepresentations. Further, for $0\leq i \leq \ell$,   holds
\begin{equation}
\label{smallest}
J.H.\left(\left({\C R}(n,d)_{(k)}^{(\rho)}\right)_{\text{\rm smallest}}(Z(r_i(n,d)_{(k)}^{(\rho)}))\right)
=
\{Z(r_j(n,d)_{(k)}^{(\rho)}),  i\leq j \leq 
\ell \}.
\end{equation}
 Denote by $\C L$ the set of all non-zero subrepresentations of ${\C R}(n,d)_{(k)}^{(\rho)}$.  Then the mapping
 $$
J:V \mapsto \{i; Z(r_i(n,d)_{(k)}^{(\rho)})\leq V\},
$$
is a an isomorphism of the partially ordered sets
 $
\C L$  and 
$$
\{ \{\ell\}, \{\ell-1,\ell\},\dots\{0,1,\dots,\ell\}    \}.
$$
\end{proposition}

\begin{proof} We shall  first prove the claim of the proposition that
${\C R}(n,d)_{(k)}^{(\rho)}$ are  multiplicity one representations of length $\ell+1$  with the minimal lattice of subrepresentations, and that  formulas \eqref{smallest} holds.
The proof of this claim  will go by induction with respect to $d$ (with $n$ and $k$ fixed). We break the induction into two parts. These parts  follow the proofs of claims (3) and (4) of Proposition \ref{prop-diff}.

The first of the induction is  for indexes $d$ which satisfy 
$$
d\leq k.
$$
In this situation $n+d-k=\min(n,n+d-k)$.

If $d=k-n$, then ${\C R}(n,d)_{(k)}^{(\rho)}$ is irreducible and isomorphic to $Z(r_0(n,d)_{(k)}^{(\rho)})$. Therefore, the claim on the composition series holds and we have the basis of the induction.

Suppose $d=1$ and $d>k-n$ (i.e. $1>k-n$). Then the assumptions on the indexes imply $k=n$. Now it is well known fact that
${\C R}(n,d)_{(k)}^{(\rho)}$ is a length two representation which is not semi simple. Further, $Z(r_0(n,d)_{(k)}^{(\rho)})$ is an irreducible quotient of it.  This implies that  the claim hold also for $d=1$.
Therefore, it is enough to consider the case 
$$
d>1.
$$

Remark \ref{rm-der-rep} implies that for $d\geq 2$, the highest derivative of the representation ${\C R}(n,d)_{(k)}^{(\rho)}$ is isomorphic to
$$
\nu^{-1/2}{\C R}(n,d-1)_{(k)}^{(\rho)}.
$$
Let ${\C R}(n,d)_{(k)}^{(\rho)}$ be a representation of $G_m$, and let its highest derivative be a  representation of $G_{m-p}$

Recall that 
$\text{h.d.}(Z(r_j(n,d)_{(k)}^{(\rho)}))=\nu^{-1/2}Z(r_j(n,d-1)_{(k)}^{(\rho)})$
for  
$0\leq j\leq n+d-k-1$.

Observe that
$\Psi^-\circ (\Phi^-)^{p-1}$ is surjective. Further, if we consider the action of this functor on the irreducible subquotients of ${\C R}(n,d)_{(k)}^{(\rho)}$, it sends only $Z(r_{n+d-k}(n,d)_{(k)}^{(\rho)})$ to $)$. We can factor 
$$
\varphi:{\C R}(n,d)_{(k)}^{(\rho)}/Z(r_{n+d-k}(n,d)_{(k)}^{(\rho)}) \ra \nu^{-1/2}{\C R}(n,d-1)_{(k)}^{(\rho)}.
$$
 Recall that $\Psi^-\circ (\Phi^-)^{p-1}$ is an exact functor. This (together with the fact that the highest derivative carries irreducible representations to the irreducible ones) implies that $\varphi$ carries the compassion series to the composition series (i.e. if $\s_1,\dots,\s_m$ is a composition series of a subrepresentation $\pi$, then $\varphi(\s_1),\dots,\varphi(\s_m)$ is a composition series of  $\varphi(\pi)$. This (together with the multiplicity one) implies that $\varphi$ is injective mapping considered on the lattice of subrepresentations of ${\C R}(n,d)_{(k)}^{(\rho)}/Z(r_{n+d-k}(n,d)_{(k)}^{(\rho)})$. The inductive assumption implies that the quotient
$
{\C R}(n,d)_{(k)}^{(\rho)}/Z(r_{n+d-k}(n,d)_{(k)}^{(\rho)})
$
has at most $n+d-k$ non-zero subrepresentations  (since $\nu^{-1/2}{\C R}(n,d-1)_{(k)}^{(\rho)}$ satisfies this property). This immediately  implies that ${\C R}(n,d)_{(k)}^{(\rho)}/Z(r_{n+d-k}(n,d)_{(k)}^{(\rho)})$ is a representation with a  minimal lattice of subrepresentations.

Further, the uniqueness of an irreducible subrepresentation of ${\C R}(n,d)_{(k)}^{(\rho)}$  implies that this representation has at most $n+d-k+1$ non-zero subrepresentations. Since $n+d-k+1$ is its length, we conclude that  ${\C R}(n,d)_{(k)}^{(\rho)}$ has the minimal lattice of subrepresentations.
 
 Applying $\Psi^-\circ (\Phi^-)^{p-1}$ (or $\varphi$) and using  the inductive assumption, we   get that for a non-zero subrepresentation $V$ of ${\C R}(n,d)_{(k)}^{(\rho)}$  holds
\begin{equation}
\label{implication-Z}
Z(r_j(n,d)_{(k)}^{(\rho)})\in J.H.(V)\implies Z(r_i(n,d)_{(k)}^{(\rho)})\in J.H.(V) \quad \text{for all $j\leq i \leq n+d-k$}
\end{equation}
 This finishes the proof of the inductive step for the first part.

For the second part of the induction, i.e. when $k<d$, we proceed similarly. For $d=k$ we know that the claim holds from the first part of the induction which we have proved. Observe that in this case we have  $n=\min(n,n+d-k)$ (since $k\leq d$). The proof of the inductive step for this case follows the proof of the first part (actually, it is simpler, since we do not need to go to a quotient of ${\C R}(n,d)_{(k)}^{(\rho)}$).
One concludes directly the claim inductively, using the formula for  the highest derivative and the minimality of the lattice of subrepresentations of the highest derivative (which is the inductive assumption). We get also \eqref{implication-Z} from the inductive assumption. This finishes the proof of the claim in this case.

The claim of the proposition that we have proved, directly implies the claim of the proposition on the  on the lattice $\C L$ of subrepresentations. This finishes the proof of the proposition.
\end{proof}

\section{Composition series - non-disjoint beginnings of segments}

We continue with the notation of the previous sections. In this section we shall assume that
$$
A_+\leq C_-.
$$
 Passing to the $n,d,k$-notation, this becomes 
$-\frac{d-1}2-\frac{n-1}2+\frac k2\leq -\frac{d-1}2+\frac{n-1}2-\frac k2 $, i.e.
$$
k< n.
$$

Recall that  for $j\ne 0$,  $r_j(n,d)^{(\rho)}_k$ is defined whenever $\D_n\ra \G_j$. Then we get $r_j(n,d)^{(\rho)}_k$  by replacing
in
$$
r_0(n,d)^{(\rho)}_k=(\D_{1},\dots,\D_n,\G_1,\dots,\G_n)
$$
the part
$$
\D_{n-j+1},\dots,
\D_n,\G_1,
\dots,\G_j
$$
with
$$
\D_{n-j+1}\cup \G_1,\D_{n-j+1}\cap \G_1,
\dots\dots,
\D_{n}\cup\G_j,\D_n\cap\G_j.
$$

Recall that $\D_{n}\ra \G_{j}$ 
if and only 
$$
\max( n-k+1,1)\leq j \leq \min(n,n+d-k).
$$

\begin{proposition}
\label{prop-same}
 Let $A_+\leq C_-$, i.e.
$$
k<n.
$$
Then

\begin{enumerate} \item ${\C R}(n,d)_{(k)}^{(\rho)}$ is a multiplicity one representation.

\item  Suppose $B_-<A^+$. In that case
$$
d\leq  k<n.
$$
Then ${\C R}(n,d)_{(k)}^{(\rho)}$ has length $d+1$, and its composition series consists of 
$$
Z(r_i(n,d)_{(k)}^{(\rho)}), \quad n-k+1 \leq i \leq n-k+d,
$$
together with $Z(r_{0}(n,d)_{(k)}^{(\rho)})$.
\item Suppose $A^+\leq B_-$, i.e.
$$
k < d    .
$$
Then ${\C R}(n,d)_{(k)}^{(\rho)}$ has length $k+1$, and its composition series consists of 
$$
Z(r_i(n,d)_{(k)}^{(\rho)}), \quad n-k+1\leq i \leq n,
$$
together with $Z(r_{0}(n,d)_{(k)}^{(\rho)})$.
\end{enumerate}
\end{proposition}

\begin{proof} 
We shall first prove (2), proving in the same time also that all the multiplicities are one. This statement (with the multiplicity one claim) will be denoted by (2)$^+$.
We fix  $n$ and $k$ and prove (2)$^+$ by induction with respect to $d$. 

For $d=1$, $1\leq k<n$ implies that ${\C R}(n,1)_{(k)}^{(\rho)}$ is a multiplicity one representation of length two. In its composition series there is obviously $Z(r_0(n,1)_{(k)}^{(\rho)})$. Further, a composition factor is also $Z((a_-^t+a_+^t)^t)$, which one directly computes using MWA$^\la$, and the result is $r_{n-k+1}(n,1)_{(k)}^{(\rho)}$. This completes the proof for $d=1$. Also we have the basis of the induction.

 Suppose
$$
1< d\leq k,
$$
and suppose that (2)$^+$ holds   for  $d-1$.
Recall
$$
\text{h.d.}({\C R}(n,d)_{(k)}^{(\rho)})=\nu^{-1/2}{\C R}(n,d-1)_{(k)}^{(\rho)}
$$
for $d\geq 2$.
Also
\begin{equation}
\label{hd+}
\text{h.d.}(Z(r_j(n,d)_{(k)}^{(\rho)}))=\nu^{-1/2}Z(r_j(n,d-1)_{(k)}^{(\rho)})
\end{equation}
for  
$$
\max( n-k+1,1)\leq j \leq \min(n,n+d-k-1),
$$
and also for $j=0$ holds \eqref{hd+}. Observe that $k<n$ implies $2\leq n-k+1$. Also $d\leq k$ implies $n+d-k\leq n$.

From the first relation and the inductive assumption follows that ${\C R}(n,d)_{(k)}^{(\rho)})$ has precisely $d$ irreducible subquotients $\pi=Z(a)$ such that the cardinality of $a$ is $2n$, and that all of them have multiplicity one in ${\C R}(n,d)_{(k)}^{(\rho)}$. Now \eqref{hd+} and the inductive assumption  imply that theses subquotients are the         
 representations $Z(r_j(n,d)_{(k)}^{(\rho)})$, $n-k+1 \leq j \leq n-k+d-1$. Now Proposition \ref{prop-sub-quo}  imply also that $Z((r_{0}(d,n)_{(k)}^{(\rho)})^t)$ is an irreducible subquotient, and that the multiplicity of this subquotient  is one. We need to  show that  
$$
(r_{0}(d,n)_{(k)}^{(\rho)})^t=  r_{n-k+d}(n,d)_{(k)}^{(\rho)}.
$$
This follows in the same way as we have proved (e) in (3) of Lemma \ref{le-crucial}. We illustrate the proof  again by drawing, instead of going into technical details:
\begin{equation} 
\label{eq-ex-4}
\xymatrix@C=.6pc@R=.1pc
{ 
 \bullet& \bullet\ar @{-}[l]   
\\
&\bullet& \bullet\ar @{-}[l]   
&
\circ& \circ\ar @{-}[l]   
\\
&&\bullet& \bullet\ar @{-}[l]   
&
\circ& \circ\ar @{-}[l]   
\\
&&&\bullet& \bullet\ar @{-}[l]   
&
\circ& \circ\ar @{-}[l]  
\\
&&&&\bullet& \bullet\ar @{-}[l]   
&
\circ& \circ\ar @{-}[l]  
\\
&&&&&&&\circ& \circ\ar @{-}[l]   
} 
\end{equation}

\begin{equation} 
\label{eq-ex-5}
\xymatrix@C=.6pc@R=.1pc
{ 
 \bullet& \bullet   
\\
&\bullet\ar @{-}[lu] & \bullet\ar @{-}[lu]   
&
\circ & \circ   
\\
&&\bullet\ar @{-}[lu]& \bullet\ar @{-}[ul]   
&
\circ\ar @{-}[lu]& \circ\ar @{-}[ul]   
\\
&&&\bullet\ar @{-}[lu]& \bullet\ar @{-}[ul]   
&
\circ\ar @{-}[lu]& \circ\ar @{-}[ul]  
\\
&&&&\bullet\ar @{-}[lu]& \bullet\ar @{-}[ul]   
&
\circ\ar @{-}[lu]& \circ\ar @{-}[lu]  
\\
&&&&&&&\circ\ar @{-}[lu]& \circ\ar @{-}[lu]   
} 
\end{equation}

\begin{equation} 
\label{eq-ex-6}
\xymatrix@C=.6pc@R=.1pc
{ 
 \bullet& \bullet\ar [l]   
\\
&\bullet& \bullet\ar  [l]   
&
\circ\ar  [l]& \circ\ar  [l]   
\\
&&\bullet& \bullet\ar  [l]   
&
\circ\ar  [l]& \circ\ar  [l]   
\\
&&&\bullet& \bullet\ar  [l]   
&
\circ\ar  [l]& \circ\ar  [l]  
\\
&&&&\bullet& \bullet\ar  [l]   
&
\circ\ar  [l]& \circ\ar  [l]  
\\
&&&&&&&\circ& \circ\ar  [l]   
} 
\end{equation}

It remains to prove that these sub quotients are all the irreducible subquotients. For this, it is enough to prove that the length of ${\C R}(n,d)_{(k)}^{(\rho)}$ is $d+1$ (counting also multiplicities).  Note that we can apply (3) of Proposition \ref{prop-diff} to ${\C R}(d,n)_{(k)}^{(\rho)}$ (observe that the roles of $n$ and $d$ are switched). That proposition tells us that the length is $d+1$, which we needed.

Now we shall prove (3). As above, in the same way we introduce (3)$^+$.  
Suppose $n=1$. Then $k=0$, and the claim obviously holds.
Suppose $n>1$, and suppose that the claim holds for $n'<n$. We shall now prove (3) for this $n$
 by induction with respect to $d$. Recall that we need to prove for $k< d$. Actually, we shall prove  a slightly stronger   (3)$^+$. We shall prove it by induction for
$$
k\leq d.
$$
For $d=k$, claim (3)$^+$  holds by (2)$^+$. 
Let $k<d$, and suppose that our claim holds for $d-1$. Then in the same way as in the first part of the proof, looking at the highest derivative of ${\C R}(n,d)_{(k)}^{(\rho)}$, the inductive assumption implies that $Z(r_j(n,d)_{(k)}^{(\rho)})$, $n-k+1\leq j \leq n$, together with $Z(r_{0}(n,d)_{(k)}^{(\rho)})$, are subquotients, and that all these irreducible subquotients  have multiplicity one. 
We also know that these are all the irreducible subquotients represented by multisegments of cardinality $2n$.

Suppose $n\leq d$ (i.e. $C_-\leq B_-$).
In this case, since  the condition $A_+\leq B_-$ holds (we assume it in (3)), we can  apply  (4) of Lemma \ref{le-crucial}, which says   that these are all the irreducible subquotients. This completes the proof in this case.

Suppose $d<n$ (i.e. $B_-<C_-$). 
We have seen that we have $k+1$ irreducible subquotients. For them, we want to prove that they exhaust the composition series of ${\C R}(n,d)_{(k)}^{(\rho)}$. Obviously, it is enough to prove that ${\C R}(n,d)_{(k)}^{(\rho)}$ has length $k+1$. For this, it is enough to show that $({\C R}(n,d)_{(k)}^{(\rho)})^t={\C R}(d,n)_{(k)}^{(\rho)}$ has length $k+1$. 

Now consider ${\C R}(d,n)_{(k)}^{(\rho)}$ and denote $n'=d$, $d'=n$. In the same way as we have introduced numbers $A_\pm,B_\pm,C_\pm,D_\pm$ for the triple $n,d,k$, we introduce $A_\pm',B_\pm',C_\pm',D_\pm'$ for the triple $n',d',k$. Observe that $A_\pm'=A_\pm,B_\pm'=C_\pm,C_\pm'=B_\pm$ and $D_\pm'=D_\pm$. Then $n'<d'$ (i.e. $C_-'<B_-'$). Since $A_+\leq C_-, A^+\leq B_-$ and $B_-<C_-,$ we have  
$$
A_+'\leq B_-', A_+'\leq C_-' \text{  and  } C_-'<B_-'.
$$
Therefore, ${\C R}(n',d')_{(k)}^{(\rho)}={\C R}(d,n)_{(k)}^{(\rho)}$ is covered by (3), and the inductive assumption implies that the length of this representation is $k+1$ (recall $n'=d<n$).

The proof of the proposition is now complete.
\end{proof}

\begin{theorem}
\label{th-not-diff}
 Let $k\in\Z_{\geq 0}$.
 
 \begin{enumerate}

\item Representations   ${\C R}(n,d)_{(k)}^{(\rho)}$ and   ${\C R}(n,d)_{(-k)}^{(\rho)}$ have the same composition series. They are multiplicity one representations.
\item
 For
$$
k<n,
$$
the composition series of  ${\C R}(n,d)_{(k)}^{(\rho)}$ are given by
$$
Z(r_i(n,d)_{(k)}^{(\rho)}), \quad n-k+1\leq i \leq \min(n-k+d,n),
$$
together with $Z(r_{0}(n,d)_{(k)}^{(\rho)})$.
\item The representation ${\C R}(n,d)_{(k)}^{(\rho)}$ has a unique irreducible subrepresentation  and also a unique irreducible quotient. The irreducible quotient is isomorphic to $Z(r_{0}(n,d)_{(k)}^{(\rho)})$. The irreducible subrepresentation is isomorphic to $Z(r_{\min(n-k+d,n)}(n,d)_{(k)}^{(\rho)}).$ 

For ${\C R}(n,d)_{(-k)}^{(\rho)}$, we have opposite situation regarding irreducible subrepresentation and quotient.
\end{enumerate}
\end{theorem}

\begin{proof} It remains to prove only  the claim regarding the irreducible subrepresentation  in (3). For $d\leq k \ (<n)$, this is proved in the proof of the  previous proposition. Therefore, we shall suppose 
$$
k<d.
$$
 For this, one needs to prove $(r_0(d,n)_{(k)}^{(\rho)})^t=r_n(n,d)_{(k)}^{(\rho)}$. One can get this directly from MWA$^\la$. We can give here  also a different argument, like in the proof of Theorem \ref{th-diff}. Again
 $$
 r_0(d,n)_{(k)}^{(\rho)}=a(d+k,n)^{(\rho)}+a(d-k,n)^{(\rho)}.
 $$
 Further, as before we have 
 $$
 Z(a(d+k,n)^{(\rho)}+a(d-k,n)^{(\rho)})^t=
 Z(a(n,d+k)^{(\rho)}+a(n,d-k)^{(\rho)}).
 $$
  This implies  
 $$
 (r_0(d,n)_{(k)}^{(\rho)})^t= a(n,d+k)^{(\rho)}+a(n,d-k)^{(\rho)}.
 $$
 Direct checking gives
 $$
 r_n(n,d)_{(k)}^{(\rho)}= a(n,d+k)^{(\rho)}+a(n,d-k)^{(\rho)}.
 $$
 The proof is now complete.
\end{proof}

\section{Lattice of subrepresentations - non-disjoint beginnings of segments}
\label{sec-der-lattice-non-disjoint-b}

The following proposition is proved in a similar way as  Proposition \ref{prop-lat-diff}. 
Therefore, we omit the proof here (the proof proceeds in two parts, following  claims (2) and (3) of Proposition \ref{prop-same}, similarly as the proof of Proposition \ref{prop-lat-diff}  followed  claims (2) and (3) of Proposition \ref{prop-diff}).

\begin{proposition}
\label{prop-not-diff-lat}
 Let $k\in\Z_{\geq 0}$.
 and
$
k<n.
$
Denote $\ell=\min(n,n+d-k)$.
Then ${\C R}(n,d)_{(k)}^{(\rho)}$ is a multiplicity one representation of length $\ell-n+k+1$, and
it is a representation with the minimal lattice of subrepresentations.  Further, for $n-k+1\leq i \leq \ell$,   holds
$$
J.H.\left(\left({\C R}(n,d)_{(k)}^{(\rho)}\right)_{\text{\rm smallest}}(Z(r_i(n,d)_{(k)}^{(\rho)}))\right)
=
\{Z(r_j(n,d)_{(k)}^{(\rho)}), i\leq j \leq \ell \}.
$$
Denote by $\C L_{\text{\rm proper}}$ the set of all non-zero proper subrepresentations of ${\C R}(n,d)_{(k)}^{(\rho)}$.  Then the mapping
 $$
J:V \mapsto \{i; Z(r_i(n,d)_{(k)}^{(\rho)})\leq V\},
$$
is a an isomorphism of the partially ordered sets
 $
\C L_{\text{\rm proper}}$  and 
$$
\{ \{\ell\}, \{\ell-1,\ell\},\dots\{n-k+1,\dots,\ell-1,\ell\}    \}.
$$
\qed
\end{proposition}

Since the lattice $\C L$ of all the non-zero  subrepresentations of ${\C R}(n,d)_{(k)}^{(\rho)}$ is an union of  $\C L_{\text{\rm proper}}$ and ${\C R}(n,d)_{(k)}^{(\rho)}$, the above description of $\C L_{\text{\rm proper}}$ completely describe the lattice $\C L$. 

\section{Translation to the setting of the Langlands classification}
\label{sec-conn-Z-L}

\subsection{Connection of Zelevinsky and Langlands classifications}
In the second section we have recalled  the definition of the Zelevinsky involution, and very basic properties of it (this   was enough for our work  on getting the composition series in terms of the Zelevinsky classification). To translate our main results from the Zelevinsky classification   to the   Langlands classifications,  we shall first recall how the Zelevinsky involution relates these two classifications (following F. Rodier's paper \cite{Ro}).

First recall that originally A.V. Zelevinsky has defined a ring homomorphism
 $^t:R \ra R$  determined by the requirement that    
 $$
 \z(\D)^t=
 \d(\D), \quad \forall \D \in S
 $$
  (recall that $R$ is a polynomial algebra over  $
\z(\D), \D \in \C S$).  He has also   shown that $^t$ is an involution. As we already noted, the fundamental result is that this involution is positive, i.e. that $r\geq 0$ implies $r^t\geq 0$ (this is proved in \cite{Au}, and in \cite{ScSt}). This implies that a restriction of $^t$ to the irreducible representations is a bijection. We  recall of the following simple result of F. Rodier (\cite{Ro}):

\begin{proposition} With $^t$ defined in a such way, 
for $a\in M(\SC)$ holds
$$
Z(a)^t=L(a) \ \ \text{and} \ \  L(a)^t=Z(a).
$$
\end{proposition}

This proposition implies that the Zelevinsky original definition of the involution agrees with the one that we have used (from the second section). For the convenience of the reader, we shall recall of a very simple argument of F. Rodier showing the above relation.

\begin{proof} Write $a=(\D_1,\dots,\D_n)$.
The proof goes by induction with respect to the standard ordering on $M(\SC)$. Suppose that $a$ is  minimal with respect to this 
ordering. Then $\z(\D_1)\t \dots \t \z(\D_n)$ is irreducible, and from the definition of $^t$ follows
$$
Z(a)^t=(\z(\D_1)\t \dots \t \z(\D_n))^t=\d(\D_1)\t \dots \t \d(\D_n).
$$
Since $L(a)$ is a subquotient of the right hand side, we get $Z(a)^t=L(a)$ in this case.

Suppose now that $a$ is arbitrary, and that the formula holds for all $a'<a$. By Theorem 7.1 of \cite{Z}, there exist positive integers $m_{a,a'}$ such that $\zeta( a)=Z(a)+\sum_{a'<a}m_{a,a'}Z(a')$. Now the inductive assumption implies
$$
\lambda(a)=\z( a)^t=Z(a)^t+\sum_{a'<a}m_{a,a'}
L(a').
$$
Since $L(a)\leq \lambda(a),$ we conclude $L(a)=Z(a)^t$. This proves the first relation in the proposition. The second relation follows immediately from the first one.
\end{proof}

\begin{corollary}
Let $P\in \Z[X_1,\dots,X_n]$ be a polynomial and $a_1,\dots,a_n\in M(\SC)$. Then in $R$ holds:
$$
P(Z(a_1),\dots,Z(a_n))=0 \iff P(L(a_1),\dots,L(a_n))=0.
$$

\end{corollary}

\begin{proof}
Write 
$$
P=\sum_{i_1,\dots,i_n\geq 0}
c_{i_1,\dots,i_n}X_1^{i_1}\dots X_n^{i_n}.
$$
Recall that $^t$ is a bijection (since it is an involution). Thus
$$
P(Z(a_1),\dots,Z(a_n))=0\iff
P(Z(a_1),\dots,Z(a_n))^t=0
$$
$$
\iff
(\sum_{i_1,\dots,i_n\geq 0}
c_{i_1,\dots,i_n}Z(a_1)^{i_1}\dots Z(a_n)^{i_n})^t=0
$$
$$
\iff
\sum_{i_1,\dots,i_n\geq 0}
c_{i_1,\dots,i_n}(Z(a_1)^t)^{i_1}\dots (Z(a_n)^t)^{i_n}=0
$$
$$
\iff
\sum_{i_1,\dots,i_n\geq 0}
c_{i_1,\dots,i_n}L(a_1)^{i_1}\dots L(a_n)^{i_n}=0
\iff
 P(L(a_1),\dots,L(a_n))=0.
$$
This completes the proof.
\end{proof}

\begin{definition} We denote
$$
\mathbf R^{\mathbf t}(n,d)_{(l)}^{(\rho)}:=\nu^{-l/2}L(a(n,d)^{(\rho)})\times \nu^{l/2}L(a(n,d)^{(\rho)})
$$
$$
\hskip19mm =L(a(n,d)^{(\nu^{-l/2}\rho)})\times L(a(n,d)^{(\nu^{l/2}\rho)}),
$$
where $n,d \in\mathbb Z_{\geq1}$, $l\in\Z$ and $\rho\in\C C$. 
\end{definition}

Observe that
 \eqref{t} implies
\begin{equation}
\label{connection}
\mathbf R^{\mathbf t}(n,d)_{(l)}^{(\rho)}= Z(a(d,n)^{(\nu^{-l/2}\rho)})\times Z(a(d,n)^{(\nu^{l/2}\rho)})
=
{\C R}(d,n)_{(l)}^{(\rho)}.
\end{equation}

Now from Theorem \ref{th-diff} easily follows the following completely analogous theorem in the setting of the Langlands classification\footnote{The only significant difference in the formulation of the theorem below is the description of  the parameters of the unique irreducible subrepresentation and the irreducible quotient.}:

\begin{theorem}
\label{th-diff-L}
 Let $k\in\Z_{\geq 0}$. 
\begin{enumerate}

\item The representations   $\mathbf R^{\mathbf t}(n,d)_{(k)}^{(\rho)}$ and   $\mathbf R^{\mathbf t}(n,d)_{(-k)}^{(\rho)}$ have the same composition series. They are multiplicity one representations.
\item
For 
$$
 n+d\leq k
 ,
 $$
   $\mathbf R^{\mathbf t}(n,d)_{(k)}^{(\rho)}$ is irreducible, and $\mathbf R^{\mathbf t}(n,d)_{(k)}^{(\rho)}\cong L(r_0(n,d)_{(k)}^{(\rho)})$.

\item
 For 
 $$
n\leq k\leq n+d,
$$
 $\mathbf R^{\mathbf t}(n,d)_{(k)}^{(\rho)}$ is a multiplicity one representation. Its composition series consists of
$$
L(r_i(n,d)_{(k)}^{(\rho)}), \quad 0\leq i \leq \min(n,n+d-k). 
$$

\item The representation $\mathbf R^{\mathbf t}(n,d)_{(k)}^{(\rho)}$ has a unique irreducible subrepresentation as well as a unique irreducible quotient. The irreducible subrepresentation  is isomorphic to $L(r_0(n,d)_{(k)}^{(\rho)})$, and the irreducible quotient is isomorphic to $L(r_{\min(n,n+d-k)}(n,d)_{(k)}^{(\rho)}).$ 

For $\mathbf R^{\mathbf t}(n,d)_{(-k)}^{(\rho)}$, we have opposite situation regarding irreducible subrepresentation and quotient.
\end{enumerate}
\end{theorem}

\begin{proof} Theorem \ref{th-diff} and \eqref{connection} directly imply (1).

To prove (2)\footnote{We could  prove this statement easily using the principal properties of the Langlands classification. Rather, we present here  the proof which follows general principle how one lifts a result from one classification to the other classification.}, let $P$ be the polynomial $X_2X_3-X_1$ and  take 
$$
X_1=r_0(n,d)_{(k)}^{(\rho)},
\qquad
X_2=a(n,d)^{(\nu^{-l/2}\rho)}, 
\qquad
X_3=a(n,d)^{(\nu^{l/2}\rho)}.
$$ 
Now (2) of Theorem \eqref{th-diff} implies 
$$
P(Z(r_0(n,d)_{(k)}^{(\rho)}), Z(a(n,d)^{(\nu^{-l/2}\rho)}), Z(a(n,d)^{(\nu^{l/2}\rho)}))=0.
$$
The above corollary implies 
$$
P(L(r_0(n,d)_{(k)}^{(\rho)}), L(a(n,d)^{(\nu^{-l/2}\rho)}), L(a(n,d)^{(\nu^{l/2}\rho)}))=0,
$$
i.e.
$L(r_0(n,d)_{(k)}^{(\rho)})= L(a(n,d)^{(\nu^{-l/2}\rho)})\t L(a(n,d)^{(\nu^{l/2}\rho)})=\mathbf R^{\mathbf t}(n,d)_{(k)}^{(\rho)}$, which is the claim of (2)

We shall now comment the proof of (3). Denote $\ell=\min(n,n+d-k)$.
Take now 
$$
P=X_{\ell+1}X_{\ell+2}-X_0- \dots -X_\ell,
$$
$$
X_i=r_i(n,d)_{(k)}^{(\rho)}, \quad 0\leq i\leq \ell,
$$
$$
X_{\ell+1}=a(n,d)^{(\nu^{-l/2}\rho)}, 
\qquad
X_{\ell+2}=a(n,d)^{(\nu^{l/2}\rho)}.
$$ 
Now (3) of Theorem \ref{th-diff} implies 
$$
P(Z(r_0(n,d)_{(k)}^{(\rho)}), \dots ,Z(r_\ell(n,d)_{(k)}^{(\rho)}), Z(a(n,d)^{(\nu^{-l/2}\rho)}), Z(a(n,d)^{(\nu^{l/2}\rho)}))=0.
$$
The above corollary implies
$$
P(L(r_0(n,d)_{(k)}^{(\rho)}), \dots ,L(r_\ell(n,d)_{(k)}^{(\rho)}), L(a(n,d)^{(\nu^{-l/2}\rho)}), L(a(n,d)^{(\nu^{l/2}\rho)}))=0,
$$
which tells that that in $R$ we have
$$
L(r_0(n,d)_{(k)}^{(\rho)})+ \dots +L(r_\ell(n,d)_{(k)}^{(\rho)})= L(a(n,d)^{(\nu^{-l/2}\rho)})\t L(a(n,d)^{(\nu^{l/2}\rho)}))=\mathbf R^{\mathbf t}(n,d)_{(k)}^{(\rho)}.
$$
This  is the claim of (3).

Recall $\mathbf R^{\mathbf t}(n,d)_{(l)}^{(\rho)}= {\C R}(d,n)_{(l)}^{(\rho)}.$ Therefore, $\mathbf R^{\mathbf t}(n,d)_{(k)}^{(\rho)}$ has a unique irreducible quotient and  a unique irreducible  subrepresentation. They are   
$Z(r_0(d,n)_{(k)}^{(\rho)})$ and  $Z(r_{\min(n,n+d-k)}(d,n)_{(k)}^{(\rho)})$ respectively, i.e. $L((r_0(d,n)_{(k)}^{(\rho)})^t)$ and  $L((r_{\min(n,n+d-k)}(d,n)_{(k)}^{(\rho)})^t)$ respectively. We have seen that they are respectively
 $L(r_{\min(n,n+d-k)}(n,d)_{(k)}^{(\rho)})$ and $L(r_0(n,d)_{(k)}^{(\rho)})$. The proof is now complete.
\end{proof}

In the completely the same way we prove Theorem \ref{th-not-diff} in the setting of the Langlands classification:

\begin{theorem}
\label{th-not-diff-L}
 Let $k\in\Z_{\geq 0}$.
 
 \begin{enumerate}

\item Representations   $\mathbf R^{\mathbf t}(n,d)_{(k)}^{(\rho)}$ and   $\mathbf R^{\mathbf t}(n,d)_{(-k)}^{(\rho)}$ have the same composition series, and they are multiplicity one representations.
\item
 For
$$
k<n,
$$
the composition series of  $\mathbf R^{\mathbf t}(n,d)_{(k)}^{(\rho)}$ are given by
$$
L(r_i(n,d)_{(k)}^{(\rho)}), \quad n-k+1\leq i \leq \min(n-k+d,n),
$$
together with $L(r_{0}(n,d)_{(k)}^{(\rho)})$.
\item The representation $\mathbf R^{\mathbf t}(n,d)_{(k)}^{(\rho)}$ has a unique irreducible subrepresentation  and also a unique irreducible quotient. The irreducible subrepresentation  is isomorphic to $L(r_{0}(n,d)_{(k)}^{(\rho)})$. The irreducible quotient  is isomorphic to $L(r_{\min(n-k+d,n)}(n,d)_{(k)}^{(\rho)}).$ 

For $\mathbf R^{\mathbf t}(n,d)_{(-k)}^{(\rho)}$, we have opposite situation regarding irreducible subrepresentation and quotient.
\qed
\end{enumerate}
\end{theorem}

\subsection{Expectation}
A following natural question after this paper (and Theorem \ref{th-i}), is the question of composition series when in the case of a product of two arbitrary essentially Speh representations\footnote{The case of the "first" reducibility point may be interesting for the problem of unitarizability for classical groups. Such  reducibility is usually not complicated (with the representation of length two).}. We expect the multiplicity one to hold also here. Also, we expect the description of composition factors that we give in this paper to  hold  there in the essentially same form. More precisely, let $\pi_1=L(\D_1,\dots,\D_n)$ and $\pi_2=L(\G_1,\dots,\G_m)$ be two essentially Speh representations (segments $\D_1$ and $\G_1$ do not need to be of the same length anymore). Choose  numerations of segments which satisfy $\D_1\ra\D_2\ra\dots\ra \D_n$ and $\G_1\ra\G_2\ra\dots\ra \G_m$. Suppose that $\pi_1\t\pi_2$ reduces (there is a simple criterion describing the reducibility in \cite{T-GL-red}). Choose a numeration of $\pi_1$ and $\pi_2$ such that $\nu^\a b(\D_1)=b(\G_1)$ for some $\a\geq 0$ (actually, then $\a$ is a positive integer). Denote by $I_{\pi_1,\pi_2}$ the set of all indexes $j\in \{1,2,\dots,m\}$ for which hold $\D_n\ra \G_j$ and $1\leq n-j-1$\footnote{This condition is automatically satisfied if $\pi_1$ is a twist of $\pi_2$, which is the case for the representations  that we study in this paper. }. Let 
$
a_{\pi_1,\pi_2}^{(0)}=(\D_{1},\dots,\D_n,\G_1,\dots,\G_m)$. For
 each $ j \in I_{\pi_1,\pi_2}$ denote by $a_{\pi_1,\pi_2}^{(j)}$ a multisegment which we obtain when we replace in $
a_{\pi_1,\pi_2}^{(0)}
$ 
the part
$
\D_{n-j+1},\dots,
\D_n,\G_1,
\dots,\G_j
$
with
$$
\D_{n-j+1}\cup \G_1,\D_{n-j+1}\cap \G_1,
\dots\dots,
\D_{n}\cup\G_j,\D_n\cap\G_j.
$$
Then we expect  $J.H.(\pi_1\t\pi_2)=\{L(a_{\pi_1,\pi_2}^{(j)}); j\in I_{\pi_1,\pi_2}\cup\{0\}\}$.

Regarding the proofs, we expect that the same strategy can be used, except that we can not use the "symmetry", which is a topic of Lemma \ref{le-sym}. This symmetry is used to transfer an information that we get on ends of segments defining irreducible subquotients  using the derivatives of essentially Speh representations, to get the corresponding  information about beginnings of these segments. In the general case, there is no such a symmetry. Instead, we expect that the derivatives considered in  Remark
\ref{rm-der-dual} will provide with the corresponding information (this is equivalent to  passing to the contragredient setting).

\end{document}